\documentclass{amsart}

\usepackage{hyperref}
\usepackage[shortcuts]{extdash}
\usepackage{thmtools}
\usepackage{amssymb}
\usepackage{mathtools}
\usepackage{eucal}
\usepackage{tikz-cd}
\usepackage{cleveref}
\usetikzlibrary{positioning,fit,shapes.geometric}
\usepackage[all,cmtip,2cell]{xy}
\UseAllTwocells
\usepackage{amsmath}
\usepackage{amsfonts}
\usepackage{mathrsfs}
\usepackage{oplotsymbl}

\makeatletter
\renewcommand\operator@font{\sf}
\makeatother

\DeclareMathOperator{\Ac}{Ac}
\DeclareMathOperator{\Ch}{Ch}
\DeclareMathOperator{\Mod}{mod}
\DeclareMathOperator{\proj}{proj}
\DeclareMathOperator{\Hom}{Hom}
\DeclareMathOperator{\End}{End}
\DeclareMathOperator{\id}{id}

\numberwithin{equation}{section}

\declaretheorem[sibling=equation]{theorem}
\declaretheorem[sibling=equation]{proposition}
\declaretheorem[sibling=equation]{lemma}
\declaretheorem[sibling=equation]{corollary}
\declaretheorem[style=definition,sibling=equation]{definition}
\declaretheorem[style=remark,sibling=equation]{remark}

\title[A cohomology operation encoding the octahedral axiom]{A cohomology operation encoding the octahedral axiom}

\author[J.~Díaz Cabrera]{Javier Díaz Cabrera}
\author[F.~Muro]{Fernando Muro}
\address{Universidad de Sevilla,
Facultad de Matemáticas,
Departamento de Álgebra,
Calle Tarfia s/n,
41012 Sevilla, Spain}
\email[J.~Díaz Cabrera]{jdiaz7@us.es}
\email[F.~Muro]{fmuro@us.es}
\urladdr[F. Muro]{https://personal.us.es/fmuro/}

\begin{document}

\begin{abstract}
We introduce a new secondary cohomology operation, the \emph{naturality defect product}, associated with a graded natural transformation between the cohomology functors of two differential graded functors, and study its main properties. We then prove that a pretriangulated category satisfies the octahedral axiom if and only if the naturality defect product associated with the natural isomorphism given by a theorem of Heller vanishes on the homotopy category of acyclic complexes over it. This characterization extends immediately to $n$-angulated categories. As an application, we exhibit non-vanishing instances for the recently found examples of pretriangulated categories which are not triangulated. Finally, we show that horizontal composition of these natural isomorphisms preserves $n$-angulated structures, and we use it to obtain, from each of these examples and for every $r\geq 1$, a pre-$3r$-angulated category which is not $3r$-angulated.
\end{abstract}

\maketitle

\tableofcontents

\section{Introduction}

Triangulated categories were introduced in the early 1960s by D. Puppe \cite{puppe_1962_homotopietheorie} in the context of algebraic topology and by J.-L. Verdier \cite{verdier_1996_des_categories_derivees} in algebraic geometry. Since then, they have become a fundamental tool not only in those fields, but also in other areas such as representation theory, and have been studied extensively in the literature.
Verdier's axiomatization, however, included an additional axiom absent from Puppe's, the so-called \emph{octahedral axiom}.

As usual, we call \emph{pretriangulated categories} those satisfying Puppe's axioms, and reserve the term \emph{triangulated categories} for those that also satisfy the octahedral axiom. During the preparation of this paper, the first examples of pretriangulated categories that are not triangulated have been found; see \cite{Che26v2}, \cite{Name_Name_2026} or \cite{Ping_He-Bo_Le_2026}, and \cite{He_He_Zhou_2026} for a higher analogue.

Let $\mathscr{T}$ be an additive category with split idempotents and $\Sigma$ an autoequivalence of $\mathscr{T}$ which is the suspension functor of some pretriangulated structure. Consider the stable category of finitely presented right $\mathscr{T}$-modules $\underline{\Mod}(\mathscr{T})$. This category is triangulated with inverse suspension functor $\Omega$, the syzygy functor. In \cite[Thm.~16.4]{heller1968stable}, A. Heller proved that endowing $(\mathscr{T},\Sigma)$ with a pretriangulation is equivalent to giving a triangulated natural isomorphism between the triangulated endofunctors of $\underline{\Mod}(\mathscr{T})$ given by $\Omega^{-3}$ and the extension of $\Sigma$. 

 One could ask how the octahedral axiom translates in this setting. Sufficient conditions have been obtained in \cite{amiot_2007_structure_triangulated_categories} and \cite{muro_2020_first_obstructions_enhancing}. However, these conditions are not necessary. An example is the exotic structure described in \cite{Diaz_Muro_2026}, which does not satisfy either of them, but does satisfy the octahedral axiom as proved in \cite{Ping_He-Bo_Le_2026}.  We give a necessary and sufficient condition in terms of a secondary cohomology operation.

 The first example of a secondary cohomology operation is the triple Massey product \cite{massey_1958_some_higher_order_cohomology_operations}, which plays a fundamental role in algebraic topology; see \cite{massey_1969_higher_order_linking_numbers} for an application to link theory. Since its introduction, many other secondary and higher cohomology operations have been discovered and their use has been extended to other areas as well. 
 
 In this paper, we introduce a new secondary cohomology operation, named the \emph{naturality defect product}, associated with a graded natural transformation $\delta\colon H^*(F) \Rightarrow H^*(G)$ between the cohomology functors of two differential graded (dg) functors $F,G\colon \mathscr{A}\to \mathscr{B}$ between dg categories. This operation, which encodes the failure of naturality of $\delta$ before taking cohomology classes, takes as input two classes $f\in H^{*}(\mathscr{A})(A,B)$ and $g\in H^{*}(\mathscr{A})(B,C)$ of morphisms which compose to zero and gives as output a coset
 $$ \langle g,f\rangle_{\delta} \in \frac{H^{*}(\mathscr{B})(F(A),G(C))}{G(g)\circ H^{*}(\mathscr{B})(F(A),G(B)) + H^{*}(\mathscr{B})(F(B),G(C)) \circ F(f) }.$$
 
 To fit this operation into the triangulated setting, we pass from $\underline{\Mod}(\mathscr{T})$ to the equivalent homotopy category of acyclic complexes over $\mathscr{T}$, denoted $\underline{\Ac}(\mathscr{T})$, which has a canonical dg enhancement. This way, a pretriangulation on $\mathscr{T}$ has an associated natural isomorphism $\delta\colon \Sigma \Rightarrow [3]$ between the endofunctors of $\underline{\Ac}(\mathscr{T})$ induced by two dg functors, the suspension $\Sigma$ defined componentwise and the third power of the shift of complexes.
 
 Our main result is the following. For brevity, we only write one of several equivalent conditions.

 \begin{theorem}[{see Theorem \ref{main_theorem}}]
    Let $\mathscr{T}$ be a pretriangulated category with split idempotents and associated natural isomorphism $\delta$. Then, the following are equivalent:
    \begin{enumerate}
        \item $\mathscr{T}$ is triangulated.
        \item The naturality defect product $\langle g,f\rangle_{\delta}$ vanishes for every pair of morphisms $f,g$ in $\underline{\Ac}(\mathscr{T})$ such that $gf=0$.
    \end{enumerate}
 \end{theorem}

 This paper is organized as follows. In Section \ref{sec:operation}, we define the naturality defect product and study its properties. In Section \ref{sec:heller}, we review Heller's theory and take it to the category $\underline{\Ac}(\mathscr{T})$. We prove our main result in Section \ref{sec:octahedral}.
 In Section \ref{sec:examples_section}, we apply our theory to the examples of \cite{Che26v2} and \cite{Name_Name_2026} by exhibiting non-vanishing naturality defect products.
 Finally, in Section \ref{sec:n-angulated}, we see that the results in Section \ref{sec:octahedral} generalize to the theory of $n$-angulated categories. We also show that horizontal composition of the associated natural isomorphisms preserves $n$-angulated structures, and use it to produce, for every $r\geq 1$, pre-$3r$-angulated categories which are not $3r$-angulated.

\addtocontents{toc}{\protect\setcounter{tocdepth}{1}}
\subsection*{Acknowledgements}

Both authors were partially supported by the project PID2024-157173NB-I00 funded by MCIN\slash AEI\slash 10.13039\slash 501100011033 and by FEDER, UE. The first author acknowledges the Spanish Ministry of Science, Innovation and Universities for financial support through the FPU predoctoral grant FPU23/01226.

The first author would like to thank Gustavo Jasso for suggesting the generalization to the $n$-angulated setting during the author's visit to the University of Cologne. 

\subsection*{Declaration of AI use}

This paper was prepared with some degree of AI assistance. In particular, Claude Opus 5.5 helped with some of the computations in Section \ref{sec:pi_d4} and contributed to the statements and proofs of part $(4)$ of Proposition \ref{product_props}, the last part of Proposition \ref{horizontal_props}, and Propositions \ref{lightning_flash} and \ref{gemstone_obstruction}. It also helped in the writing and revision of the manuscript (Sonnet 5.5 was also used for this last part), including the correction of typos and of some signs. All mathematical ideas and the rest of the proofs are due to the authors, who have verified every statement and proof and take full responsibility for them.

\addtocontents{toc}{\protect\setcounter{tocdepth}{3}}
\section{The operation} \label{sec:operation}

\subsection{Definition}

Our purpose in this section is to define a new secondary cohomology operation. Let $\mathscr{A},\mathscr{B}$ be dg categories, $F,G\colon \mathscr{A}\to \mathscr{B}$ two dg functors and $\delta\colon H^*(F) \Rightarrow H^*(G)$ a graded natural transformation of degree $t$. Fix a representative $\tilde{\delta}_X$ for every $\delta_X\in H^t(\mathscr{B})(F(X),G(X))$. For every homogeneous morphism $\alpha\in \mathscr{A}(A,B)$, define
$$D_{\tilde{\delta}}(\alpha):= \tilde{\delta}_B F(\alpha) - (-1)^{t|\alpha|} G(\alpha) \tilde{\delta}_A\colon F(A)\to G(B).$$
Notice that $D_{\tilde{\delta}}(\alpha)$ encodes the failure of graded commutativity between $\tilde{\delta}$ and $\alpha$. In particular, if $\alpha$ is a cocycle, then $D_{\tilde{\delta}}(\alpha)$ is a coboundary by naturality of $\delta$.

\begin{lemma}
    $D_{\tilde{\delta}}$ is a degree $t$ derivation compatible with the differentials, that is, for every pair of composable homogeneous morphisms $\alpha,\beta$ in $\mathscr{A}$, we have
    $$D_{\tilde{\delta}}(\beta\alpha) = D_{\tilde{\delta}}(\beta)F(\alpha) + (-1)^{t|\beta|}G(\beta) D_{\tilde{\delta}}(\alpha) \hspace{4mm} \text{and} \hspace{4mm} D_{\tilde{\delta}}(d(\alpha))=(-1)^td(D_{\tilde{\delta}}(\alpha)). $$
\end{lemma}

\begin{proof}
    Let $\alpha\colon A\to B$ and $\beta\colon B\to C$ be composable morphisms of degree $p$ and $q$ respectively. Then, we have
\begin{align*}
    D_{\tilde{\delta}}(\beta\alpha) & = \tilde{\delta}_C F(\beta\alpha) - (-1)^{t(p+q)} G(\beta\alpha)\tilde{\delta}_A \\
     & = \tilde{\delta}_C F(\beta)F(\alpha) - (-1)^{t(p+q)} G(\beta)G(\alpha)\tilde{\delta}_A \\
     & = \bigl( \tilde{\delta}_C F(\beta) - (-1)^{tq} G(\beta)\tilde{\delta}_B \bigr) F(\alpha) \\
     &\qquad + (-1)^{tq} G(\beta)\bigl( \tilde{\delta}_B F(\alpha) - (-1)^{tp} G(\alpha)\tilde{\delta}_A\bigr) \\
     & = D_{\tilde{\delta}}(\beta)F(\alpha) + (-1)^{tq} G(\beta)D_{\tilde{\delta}}(\alpha),
\end{align*}
and, for the second part,
\begin{align*}
    D_{\tilde{\delta}}(d(\alpha)) & = \tilde{\delta}_B F(d(\alpha)) - (-1)^{t(p+1)} G(d(\alpha))\tilde{\delta}_A \\
     & = \tilde{\delta}_B d(F(\alpha)) - (-1)^{t(p+1)} d(G(\alpha))\tilde{\delta}_A \\
     & = (-1)^t d(\tilde{\delta}_B F(\alpha)) - (-1)^{t(p+1)} d(G(\alpha)\tilde{\delta}_A) \\
     & = (-1)^t d(\tilde{\delta}_B F(\alpha) - (-1)^{tp} G(\alpha)\tilde{\delta}_A) \\
     & = (-1)^t d(D_{\tilde{\delta}}(\alpha)).
\end{align*}
\end{proof}

\begin{definition} \label{delta-product}
    Let $F,G\colon \mathscr{A}\to \mathscr{B}$ be dg functors and $\delta$ a graded natural transformation from $H^*(F)$ to $H^*(G)$ of degree $t$, with a fixed representative $\tilde{\delta}_X$ of $\delta_X$ for every object $X$ in $\mathscr{A}$. Consider classes
    $$a=[\alpha]\in H^p(\mathscr{A})(A,B), \hspace{5mm} b=[\beta]\in H^q(\mathscr{A})(B,C)$$
    with vanishing composition $ba=0$. Choose $s$ with $\beta\alpha = d(s)$ and, by naturality of $\delta$, elements $x,y$ with
    $$ D_{\tilde{\delta}}(\alpha) = d(x), \hspace{5mm} D_{\tilde{\delta}}(\beta)=d(y). $$
    We define the cocycle
    $$ m_{\tilde{\delta}}(\alpha,\beta,s,x,y) := D_{\tilde{\delta}}(s) - (-1)^t \big( yF(\alpha) + (-1)^{q+tq}G(\beta) x \big), $$
    and the \emph{indeterminacy} subgroup
    $$ M_{\delta}(b,a) := G(b)\circ H^{p+t-1}(\mathscr{B})(F(A),G(B)) + H^{q+t-1}(\mathscr{B})(F(B),G(C))\circ F(a). $$
    Then, the \emph{naturality defect product} of $a$ and $b$ with respect to $\delta$ is the coset
    $$\langle b,a\rangle_{\delta} := \big[ \; m_{\tilde{\delta}}(\alpha,\beta,s,x,y) \; \big] \in \frac{H^{p+q+t-1}(\mathscr{B})(F(A),G(C))}{M_{\delta}(b,a)}.$$
    
\end{definition}

We say that $\langle b,a\rangle_{\delta}$ \emph{vanishes} if it is zero in the quotient by $M_{\delta}(b,a)$ or, equivalently, if $\langle b,a\rangle_{\delta}$ contains the zero cohomology class.
Our situation is shown in the following diagram:
$$\begin{tikzcd}[column sep=3.6em, row sep=large]
F(A) \arrow[r, "F(\alpha)"] \arrow[d, "\tilde{\delta}_A"'] \arrow[rr, bend left=50, "0", ""{name=T, below}] \arrow[dr, Rightarrow, shorten <=2.5mm, shorten >=2.5mm, "x"]
& F(B) \arrow[r, "F(\beta)"] \arrow[d, "\tilde{\delta}_B"'] \arrow[dr, Rightarrow, shorten <=2.5mm, shorten >=2.5mm, "y"]
& F(C) \arrow[d, "\tilde{\delta}_C"] \\
G(A) \arrow[r, "G(\alpha)"'] \arrow[rr, bend right=50, "0"', ""{name=B, above}]
& G(B) \arrow[r, "G(\beta)"']
& G(C)
\arrow[Rightarrow, from=1-2, to=T, shorten <=2mm, shorten >=0.5mm, "F(s)"']
\arrow[Rightarrow, from=2-2, to=B, shorten <=2mm, shorten >=0.5mm, "G(s)"]
\end{tikzcd}$$
Here the double arrows are morphisms of degree one less whose differentials are, up to sign, the difference between the two composites around the cell in which they sit. Then $m_{\tilde{\delta}}(\alpha,\beta,s,x,y)$ is the sum, with appropriate signs, of the four morphisms $F(A)\to G(C)$ obtained by composing each double arrow with the arrows of the diagram: $\tilde{\delta}_C F(s)$, $G(s)\tilde{\delta}_A$, $yF(\alpha)$ and $G(\beta)x$.

\begin{proposition}
    $\langle b,a\rangle_{\delta}$ is well defined and does not depend on any choice.
\end{proposition}

\begin{proof}
        Let $\alpha,\beta,\tilde{\delta},s,x$ and $y$ be as in Definition \ref{delta-product}. First, we check that $m_{\tilde{\delta}}(\alpha,\beta,s,x,y) $ is a cocycle. Indeed,
        \begin{align*}
            d(m_{\tilde{\delta}}(\alpha,\beta,s,x,y)) & =  d(D_{\tilde{\delta}}(s)) - (-1)^t(d(yF(\alpha)) + (-1)^{q+tq}d(G(\beta)x)) \\
             & =  (-1)^t D_{\tilde{\delta}}(d(s)) - (-1)^t(d(y)F(\alpha)+(-1)^{tq}G(\beta)d(x)) \\
             & =  (-1)^t \big( D_{\tilde{\delta}}(\beta\alpha) - ( D_{\tilde{\delta}}(\beta)F(\alpha) + (-1)^{tq} G(\beta)D_{\tilde{\delta}}(\alpha))\big) \\
            & =  0.
        \end{align*}

        We now prove independence of every possible choice. Let us start with the choice of $\alpha$. Let $ [\alpha']=a= [\alpha]$, in particular $\alpha'=\alpha+d(\varepsilon)$.
        We choose $x'=x+ (-1)^{t} D_{\tilde{\delta}}(\varepsilon)$ and $s'=s+ (-1)^q\beta \varepsilon$. These satisfy
        $$d(x') = d(x) +(-1)^t d(D_{\tilde{\delta}}(\varepsilon)) = D_{\tilde{\delta}}(\alpha)+ D_{\tilde{\delta}}(d(\varepsilon)) = D_{\tilde{\delta}}(\alpha') $$
        $$d(s') = d(s)+ (-1)^q d(\beta \varepsilon)= \beta\alpha + \beta d(\varepsilon)= \beta \alpha' $$
        and the cocycle
        \begin{align*}
            & m_{\tilde{\delta}}(\alpha',\beta,s',x',y)- m_{\tilde{\delta}}(\alpha,\beta,s,x,y) \\
            &= D_{\tilde{\delta}}(s'-s) - (-1)^{t} \Big(yF(\alpha'-\alpha) + (-1)^{q+tq}G(\beta)(x'- x) \Big) \\
            &= D_{\tilde{\delta}}((-1)^q \beta \varepsilon) -(-1)^t \Big(yF(d(\varepsilon)) +(-1)^{q+tq+t}G(\beta)D_{\tilde{\delta}}(\varepsilon) \Big)  \\
            &= (-1)^q D_{\tilde{\delta}}( \beta \varepsilon) - (-1)^{q+tq} G(\beta)D_{\tilde{\delta}}(\varepsilon) - (-1)^t yF(d(\varepsilon))  \\
            &= (-1)^q \tilde{\delta}_C F(\beta \varepsilon) - (-1)^{q+t(q+p-1)} G(\beta \varepsilon)\tilde{\delta}_A - (-1)^{q+tq} G(\beta)\tilde{\delta}_B F(\varepsilon)  \\
            &\qquad + (-1)^{q+t(q+p-1)} G(\beta)G(\varepsilon)\tilde{\delta}_A - (-1)^{t} yF(d(\varepsilon)) \\
            &= (-1)^qD_{\tilde{\delta}}(\beta)F(\varepsilon) - (-1)^t yd(F(\varepsilon)) \\
            &= (-1)^qd(y)F(\varepsilon)- (-1)^t yd(F(\varepsilon)) \\
            &= (-1)^q d(yF(\varepsilon))
        \end{align*}

    is a coboundary. The independence in choosing $\beta$ is similar.

    Now, for the choice of $\tilde{\delta}$, let $\tilde{\delta}'$ be another representative of $\delta$ and let $h=(h_X)_{X\in \mathscr{A}} $ be such that $\tilde{\delta}'=\tilde{\delta}+d(h)$ componentwise. Let $\alpha, \beta,x,y$ and $s$ be as above. Then, we can choose
    $$x' = x + h_BF(\alpha) - (-1)^{tp+p}G(\alpha)h_A , \hspace{5mm} y'= y+h_CF(\beta)- (-1)^{tq+q}G(\beta)h_B $$
    so that they satisfy
    $$\begin{array}{ccl}
            d(x') & = &  d(x) +d(h_BF(\alpha)) - (-1)^{tp+p} d(G(\alpha)h_A) \\
            & = &  D_{\tilde{\delta}}(\alpha) + d(h_B)F(\alpha) - (-1)^{tp} G(\alpha)d(h_A) \\
            & = & \tilde{\delta}_B F(\alpha) -(-1)^{tp}G(\alpha)\tilde{\delta}_A  + d(h_B)F(\alpha) - (-1)^{tp} G(\alpha)d(h_A) \\
            & = &  (\tilde{\delta}_B+d(h_B))F(\alpha) - (-1)^{tp} G(\alpha)(d(h_A)+\tilde{\delta}_A) \\
            & = & D_{\tilde{\delta}'}(\alpha),
        \end{array}$$
    similarly, $d(y')=D_{\tilde{\delta}'}(\beta)$ and the cocycle
    \begin{align*}
        &  m_{\tilde{\delta}'}(\alpha,\beta,s,x',y') - m_{\tilde{\delta}}(\alpha,\beta,s,x,y) \\
        &= (D_{\tilde{\delta}'}(s)-D_{\tilde{\delta}}(s)) -(-1)^t\Big( (y'-y)F(\alpha) + (-1)^{q+tq}G(\beta)(x'-x) \Big) \\
        &= (\tilde{\delta}'_C-\tilde{\delta}_C)F(s) - (-1)^{t(p+q-1)} G(s)(\tilde{\delta}'_A-\tilde{\delta}_A)  \\
        &\qquad -\Big((-1)^t h_CF(\beta) -(-1)^{tq+q+t} G(\beta)h_B\Big)F(\alpha)  \\
        &\qquad - G(\beta)\Big((-1)^{t+q+tq}h_BF(\alpha) - (-1)^{t+q+tq+tp+p} G(\alpha)h_A\Big)  \\
        &= d(h_C)F(s) - (-1)^{t(p+q-1)}G(s)d(h_A) \\
        &\qquad -(-1)^t h_C F(\beta\alpha) + (-1)^{t+q+tq+tp+p}G(\beta\alpha)h_A \\
        &= d(h_C)F(s) - (-1)^{t(p+q-1)}G(s)d(h_A)  \\
        &\qquad -(-1)^t h_C d(F(s)) + (-1)^{t+q+tq+tp+p}d(G(s))h_A \\
        &= d\big(h_CF(s) + (-1)^{t+q+tq+tp+p} G(s)h_A\big)
    \end{align*}
    is also a coboundary.

        We observe that $\langle b,a\rangle_{\delta}$ is independent of the choice of $s$ due to the naturality of $\delta$, which implies that $D_{\tilde{\delta}}$ applied to every cocycle is a coboundary. Indeed, if $s'$ satisfies $d(s') = \beta\alpha = d(s)$, then we get
        $$m_{\tilde{\delta}}(\alpha,\beta,s',x,y) - m_{\tilde{\delta}}(\alpha,\beta,s,x,y) = D_{\tilde{\delta}}(s') - D_{\tilde{\delta}}(s) = D_{\tilde{\delta}}(s'-s), $$
        which is a coboundary since $s'-s$ is a cocycle.

        Finally, we show that the indeterminacy $M_{\delta}(b,a)$ takes care of the choices of $x$ and $y$. We write only the choice of $y$ since the other one is analogous. Notice that, as before, any other choice $y'$ would make $y'-y$ a cocycle. Then,
    $$\begin{array}{ccl}
        \big[\, m_{\tilde{\delta}}(\alpha,\beta,s,x,y') - m_{\tilde{\delta}}(\alpha,\beta,s,x,y) \,\big] & = & \big[\, - (-1)^t (y'-y)F(\alpha) \,\big] \\ & = & -(-1)^t [y'-y] \circ F(a),
    \end{array}$$
        which is an element of $H^{q+t-1}(\mathscr{B})(F(B),G(C))\circ F(a)\subseteq M_{\delta}(b,a)$.
\end{proof}

\subsection{Properties}

We give here some basic properties of naturality defect products, starting with the following classical result on secondary cohomology operations:

\begin{proposition} \label{all_elements_come_from_choices}
    Fix $\alpha,\beta,\tilde{\delta}$ and $s$ as in Definition \ref{delta-product}. Then, every element of $\langle b,a\rangle_{\delta}$ is the cohomology class of $m_{\tilde{\delta}}(\alpha,\beta,s,x,y)$ for some choice of $x$ and $y$.
\end{proposition}
\begin{proof}
    Let $x,y$ be as in Definition \ref{delta-product}. Then, any element of $\langle b,a\rangle_{\delta} $ is of the form
    $$ [m_{\tilde{\delta}}(\alpha,\beta,s,x,y)]+ G(b)\circ [h] + [h']\circ F(a) $$
    for some cocycles $h,h'$. Operating with classes, we see that this is the same as the class of
    $$ D_{\tilde{\delta}}(s) -(-1)^t \big( (y - (-1)^th')F(\alpha) + (-1)^{q+tq}G(\beta)(x - (-1)^{t+q+tq}h) \big) $$
    which is exactly $m_{\tilde{\delta}}(\alpha,\beta,s,x',y')$ by taking
    $$x' = x - (-1)^{t+q+tq}h \hspace{4mm}
    \text{and} \hspace{4mm} y' = y - (-1)^th',$$
     and which obviously satisfy $D_{\tilde{\delta}}(\alpha) = d(x') $ and $D_{\tilde{\delta}}(\beta)=d(y') $.
\end{proof}

The following proposition shows that the naturality defect product encodes the failure of naturality of $\delta$ before taking cohomology classes.

\begin{proposition}
    Suppose $F,G\colon \mathscr{A}\to \mathscr{B}$ are dg functors and $\tilde{\delta}\colon F\Rightarrow G$ is a graded natural transformation such that $\tilde{\delta}_X$ is a cocycle for every $X\in \mathscr{A}$. Define $\delta\colon H^*(F)\Rightarrow  H^*(G) $ by $\delta_X= [\tilde{\delta}_X]$. Then, the corresponding naturality defect product vanishes wherever it is defined.
\end{proposition}

\begin{proof}
    Let $a= [\alpha], b=[\beta] $ be composable classes with $\beta\alpha = d(s)$ for some $s$. We take $\tilde{\delta}$ itself as the representative of $\delta$. Since $\tilde{\delta}$ is graded natural, we have that both $D_{\tilde{\delta}}(\alpha)$ and $D_{\tilde{\delta}}(\beta)$ vanish, so we can take both $x$ and $y$ to be zero. With these choices, we get
    $$ m_{\tilde{\delta}}(\alpha,\beta,s,x,y) = D_{\tilde{\delta}}(s) $$
    which is again zero since $\tilde{\delta}$ is graded natural.
    
\end{proof}

\begin{corollary} \label{identity_product}
    Let $F\colon \mathscr{A}\to \mathscr{B}$ be a dg functor. Every naturality defect product with respect to the identity transformation of $H^*(F)$ vanishes.
\end{corollary}

Naturality defect products are subsets of morphism groups of $H^*(\mathscr{B})$, which we can add and compose. Namely, for subsets $X,Y$ of the same group of morphisms and suitable morphisms $f,g$, we set
$$X+Y=\{x+y\mid x\in X,\ y\in Y\}, \quad f\circ X=\{f\circ x\mid x\in X\}, \quad X\circ g=\{x\circ g\mid x\in X\}.$$

\begin{proposition} \label{product_props}
    Let $F,G, I\colon \mathscr{A}\to \mathscr{B}$ be dg functors, $\delta\colon H^*(F)\Rightarrow H^*(G)$, $\varepsilon\colon H^*(G)\Rightarrow H^*(I)$ graded natural transformations. Let 
    $$A  \overset{a}{\longrightarrow} B  \overset{b}{\longrightarrow} C  \overset{c}{\longrightarrow} D $$
    be morphisms in $H^*(\mathscr{A})$. Then, when the corresponding  naturality defect products are defined,
    \begin{enumerate}
        \item $\langle b,a\rangle_{\delta}$ vanishes provided $a$, $b$ or $\delta$ vanish.
        \item $\langle b,a\rangle_{\delta}$ is linear in $a$, $b$ and $\delta$, i.e., given $a'$, $b'$ and $\delta'$ with the same source, target and degree as $a$, $b$ and $\delta$ respectively, we have the following relations between subsets of $H^*(\mathscr{B})(F(A),G(C))$:
            \begin{align*}
                \langle b,a+a'\rangle_{\delta} &\subseteq \langle b,a\rangle_{\delta} + \langle b,a'\rangle_{\delta}, \\
                \langle b+b',a\rangle_{\delta} &\subseteq \langle b,a\rangle_{\delta}+ \langle b',a\rangle_{\delta},\\
                \langle b,a\rangle_{\delta+\delta'} &= \langle b,a\rangle_{\delta} + \langle b,a\rangle_{\delta'}.
            \end{align*}
        \item $\langle b,a\rangle_{\delta}$ is compatible with composition, i.e., we have the following inclusions between subsets of $H^*(\mathscr{B})(F(A),G(D))$:
            \begin{align*}
                \langle c,b\rangle_{\delta}\circ F(a) &\subseteq \langle c,ba\rangle_{\delta}, \\
                (-1)^{|c|(|\delta|+1)}\,G(c)\circ \langle b,a\rangle_{\delta} &\subseteq \langle cb,a\rangle_{\delta},
            \end{align*}
            and the first (second) inclusion is an equality if $F(a)$ (respectively $G(c)$) is invertible.
        \item As subsets of $H^*(\mathscr{B})(F(A),G(D))$,
            $$\langle cb, a\rangle_{\delta} \cap \langle c, ba \rangle_{\delta} \neq \varnothing$$
            and, if $b$ is invertible, then $\langle cb, a\rangle_{\delta} = \langle c, ba \rangle_{\delta}$.
            \item We have the inclusion of subsets of $H^*(\mathscr{B})(F(A),I(C))$:
        $$ \varepsilon_C \circ \langle b,a\rangle_{\delta} + (-1)^{|\delta| (|a|+|b|+1)} \langle b,a\rangle_{\varepsilon}\circ\delta_A \subseteq \langle b,a\rangle_{\varepsilon\circ\delta}, $$
        which is an equality if $\delta$ or $\varepsilon$ is a natural isomorphism. If both are, the equation holds in the quotient group by $M_{\varepsilon\circ\delta}(b,a)$.
    \end{enumerate}
\end{proposition}

\begin{proof}
    We first fix representatives $\alpha,\beta,\gamma$ of $a,b,c$, representatives $\tilde{\delta}_X$ of $\delta_X$ and $\tilde{\varepsilon}_X$ of $\varepsilon_X$ for every object $X$ in $\mathscr{A}$, and $\alpha',\beta',\tilde{\delta}'$ of $a',b',\delta'$. We set degrees $|a|=p$, $|b|=q$, $|c|=r$, $|\delta|=t$ and $|\varepsilon|=t'$. Recall that $\langle b,a\rangle_{\delta}$ is a coset of $M_{\delta}(b,a)$, and similarly for the other products, and that an inclusion between two cosets of the same subgroup is an equality.
    
    Part $(1)$ follows from the definition by making the obvious choices to get the null cocycle. 
    
    Part $(2)$ follows from linearity of the differential and $D_{\tilde{\delta}}$. Let us check linearity in $a$; the case of $b$ is similar. Let $s,s'$ be such that $d(s)=\beta\alpha$ and $d(s')=\beta\alpha'$, and let $x$ be such that $d(x)=D_{\tilde{\delta}}(\alpha)$. By Proposition \ref{all_elements_come_from_choices}, every element of $\langle b,a+a'\rangle_{\delta}$ is of the form
    $$ \Big[\, D_{\tilde{\delta}}(s+s')- (-1)^t\bigl( yF(\alpha+\alpha') + (-1)^{q+tq} G(\beta) x' \bigr) \,\Big] $$
    for some $x',y$ such that $d(x') = D_{\tilde{\delta}}(\alpha+\alpha')$ and $d(y) = D_{\tilde{\delta}}(\beta)$. Then,
     \begin{align*}
        D_{\tilde{\delta}}(s+s') &- (-1)^t \bigl( yF(\alpha+\alpha') + (-1)^{q+tq} G(\beta) x' \bigr) \\
        &= m_{\tilde{\delta}}(\alpha,\beta,s,x,y) \\
        &\qquad + \Bigl( D_{\tilde{\delta}}(s') - (-1)^t \bigl( yF(\alpha') + (-1)^{q+tq} G(\beta) (x'-x) \bigr) \Bigr)
    \end{align*}
    which is in $\langle b,a\rangle_{\delta}+\langle b,a'\rangle_{\delta}$, since $d(x'-x)=D_{\tilde{\delta}}(\alpha')$ and so the second term is $m_{\tilde{\delta}}(\alpha',\beta,s',x'-x,y)$. This inclusion can be strict since the inclusion of indeterminacies $M_{\delta}(b,a+a')\subseteq M_{\delta}(b,a)+M_{\delta}(b,a')$ may be strict.

    Now let us check linearity in $\delta$. Notice that $D_{\tilde{\delta}}$ is linear in $\tilde{\delta}$. Let $s$ be such that $d(s)=\beta\alpha$. Every element of $\langle b,a\rangle_{\delta}+\langle b,a\rangle_{\delta'}$ is the class of
    $$ m_{\tilde{\delta}}(\alpha,\beta,s,x,y)+m_{\tilde{\delta}'}(\alpha,\beta,s,x',y') = m_{\tilde{\delta}+\tilde{\delta}'}(\alpha,\beta,s,x+x',y+y') $$
    for some $x,y,x',y'$ such that $d(x)=D_{\tilde{\delta}}(\alpha)$, $d(y)=D_{\tilde{\delta}}(\beta)$, $d(x')=D_{\tilde{\delta}'}(\alpha)$ and $d(y')=D_{\tilde{\delta}'}(\beta)$. Since $d(x+x')=D_{\tilde{\delta}+\tilde{\delta}'}(\alpha)$ and $d(y+y')=D_{\tilde{\delta}+\tilde{\delta}'}(\beta)$, its class lies in $\langle b,a\rangle_{\delta+\delta'}$. Hence $\langle b,a\rangle_{\delta}+\langle b,a\rangle_{\delta'}\subseteq \langle b,a\rangle_{\delta+\delta'}$. Since $M_{\delta}(b,a)$ does not depend on $\delta$, all are cosets of the same subgroup, so they coincide.

    For part $(3)$, let $s$ be such that $d(s)=\gamma\beta$. By Proposition \ref{all_elements_come_from_choices}, an arbitrary element $f$ of $\langle c,b\rangle_{\delta}\circ F(a)$ is of the form
    \begin{align*}
        f & =  \Big[\, D_{\tilde{\delta}}(s) -(-1)^t \big( zF(\beta) +(-1)^{r+tr} G(\gamma) y \big) \,\Big]\circ [F(\alpha)]   \\
         & =  \Big[\, D_{\tilde{\delta}}(s)F(\alpha) - (-1)^t \big( zF(\beta\alpha) +(-1)^{r+tr} G(\gamma) yF(\alpha) \big) \,\Big],
    \end{align*}
    for some $y,z$ such that $d(y)=D_{\tilde{\delta}}(\beta)$ and $d(z)=D_{\tilde{\delta}}(\gamma)$. Let also $x\in \mathscr{B}^{p+t-1}(F(A),G(B))$ be such that $d(x)=D_{\tilde{\delta}}(\alpha)$. Then,
    \begin{align*}
        D_{\tilde{\delta}}(s\alpha) & = D_{\tilde{\delta}}(s)F(\alpha) + (-1)^{(q+r-1)t} G(s)d(x) \\
         & = D_{\tilde{\delta}}(s)F(\alpha) + (-1)^{(q+r-1)t}(-1)^{q+r-1} \big(d(G(s)x) - d(G(s))x\big) \\
         & = D_{\tilde{\delta}}(s)F(\alpha) -(-1)^{(q+r-1)(t+1)} G(\gamma\beta)x + (-1)^{(q+r-1)(t+1)}d(G(s)x).
    \end{align*}
    We get 
    \begin{align*}
        f & = \Big[\, D_{\tilde{\delta}}(s\alpha) + (-1)^{(q+r-1)(t+1)}G(\gamma\beta) x \\
         &\qquad -(-1)^t \big( zF(\beta\alpha)+(-1)^{r+tr}G(\gamma)yF(\alpha) \big) \,\Big] \\
        & = \Big[\, D_{\tilde{\delta}}(s\alpha) - (-1)^t \big( z F(\beta\alpha) \\
         &\qquad +(-1)^{r+rt} G(\gamma)(y F(\alpha) +(-1)^{q+tq} G(\beta)x) \big) \,\Big]
    \end{align*}
    with $d(s\alpha) = d(s)\alpha = \gamma\beta\alpha$ and 
    \begin{align*}
        d\big(yF(\alpha)+ (-1)^{q+tq} G(\beta)x\big) & = d(y)F(\alpha) + (-1)^{tq} G(\beta)d(x)  \\
         & = D_{\tilde{\delta}}(\beta)F(\alpha) + (-1)^{tq} G(\beta)D_{\tilde{\delta}}(\alpha) \\
          & = D_{\tilde{\delta}}(\beta\alpha).
    \end{align*}
    This implies $f\in \langle c, ba\rangle_{\delta}$. The other inclusion is analogous, provided one keeps careful track of the signs.

    If $F(a)$ is invertible, then composing with it is a bijection, so $M_{\delta}(c,b)\circ F(a) = M_{\delta}(c,ba)$ and $\langle c,b\rangle_{\delta}\circ F(a)$ is a coset of $M_{\delta}(c,ba)$, that is, it coincides with $\langle c,ba\rangle_{\delta}$. The case of $G(c)$ is similar, the sign being irrelevant.

    For part $(4)$, let $s$ be such that $d(s)=\gamma\beta\alpha$ and let $x,y,z$ be such that $d(x)=D_{\tilde{\delta}}(\alpha)$, $d(y)=D_{\tilde{\delta}}(\beta)$ and $d(z)=D_{\tilde{\delta}}(\gamma)$. As in the proof of part $(3)$, the elements
    $$ \bar{x} = yF(\alpha)+ (-1)^{q+tq} G(\beta)x \hspace{5mm} \text{and} \hspace{5mm} \bar{y} = zF(\beta)+ (-1)^{r+tr} G(\gamma)y $$
    satisfy $d(\bar{x}) = D_{\tilde{\delta}}(\beta\alpha)$ and $d(\bar{y}) = D_{\tilde{\delta}}(\gamma\beta)$. Using these choices, we get
    \begin{align*}
        m_{\tilde{\delta}}(\beta\alpha,\gamma,s,\bar{x},z) & = D_{\tilde{\delta}}(s) - (-1)^t \big( zF(\beta\alpha) + (-1)^{r+tr}G(\gamma)yF(\alpha) \\
        &\qquad + (-1)^{(q+r)(t+1)}G(\gamma\beta)x \big) \\
        & = m_{\tilde{\delta}}(\alpha,\gamma\beta,s,x,\bar{y}),
    \end{align*}
    so this class lies in both $\langle c,ba\rangle_{\delta}$ and $\langle cb,a\rangle_{\delta}$.

    If $b$ is invertible, then so are $F(b)$ and $G(b)$, and therefore
    \begin{align*}
        G(b)\circ H^{p+t-1}(\mathscr{B})(F(A),G(B)) &= H^{p+q+t-1}(\mathscr{B})(F(A),G(C)), \\
        H^{r+t-1}(\mathscr{B})(F(C),G(D))\circ F(b) &= H^{q+r+t-1}(\mathscr{B})(F(B),G(D)),
    \end{align*}
    so that $M_{\delta}(cb,a) = M_{\delta}(c,ba)$. Hence $\langle cb,a\rangle_{\delta}$ and $\langle c,ba\rangle_{\delta}$ are cosets of the same subgroup with non-empty intersection, so they coincide.

    For part $(5)$, let $s$ be such that $d(s)=\beta\alpha$. By Proposition \ref{all_elements_come_from_choices}, an arbitrary element $f$ of the left hand side is a sum of the form
    \begin{align*}
        &\Big[\, \tilde{\varepsilon} D_{\tilde{\delta}}(s) -(-1)^{t} \big( \tilde{\varepsilon} yF(\alpha) + (-1)^{q+tq} \tilde{\varepsilon}  G(\beta) x \big) \,\Big] \\
        &\qquad + (-1)^{t(p+q+1)} \Big[\, D_{\tilde{\varepsilon}}(s)\tilde{\delta}  - (-1)^{t'} \big( y'G(\alpha) \tilde{\delta} + (-1)^{q+t'q} I(\beta)x' \tilde{\delta} \big) \,\Big]
    \end{align*}
    for some $x,y,x',y'$ such that $d(x)=D_{\tilde{\delta}}(\alpha)$, $d(y)=D_{\tilde{\delta}}(\beta)$, $d(x')=D_{\tilde{\varepsilon}}(\alpha)$ and $d(y')=D_{\tilde{\varepsilon}}(\beta)$, where we do not specify the components of the natural transformations to simplify notation. We see that
    \begin{align*}
         & (-1)^{t+q+tq} \tilde{\varepsilon} G(\beta)x + (-1)^{t(p+q+1)+t'} y'G(\alpha)\tilde{\delta}  \\
         & = (-1)^{t+tq} \big( ((-1)^q D_{\tilde{\varepsilon}}(\beta) + (-1)^{q+t'q} I(\beta)\tilde{\varepsilon} )x + (-1)^{t'} y'(\tilde{\delta} F(\alpha) - D_{\tilde{\delta}}(\alpha) ) \big)  \\
         & =  (-1)^{t+tq} \big( (-1)^{q+t'q} I(\beta)\tilde{\varepsilon} x + (-1)^{t'} y'\tilde{\delta} F(\alpha) + (-1)^q d(y')x - (-1)^{t'} y'd(x) \big) \\
         & =  (-1)^{t+tq} \big( (-1)^{q+t'q} I(\beta)\tilde{\varepsilon} x + (-1)^{t'} y'\tilde{\delta} F(\alpha)\big) + d((-1)^{t+q+tq}y'x) .
    \end{align*}
    It follows that $f$ equals the class of
    \begin{align*}
        &\tilde{\varepsilon} D_{\tilde{\delta}}(s) + (-1)^{t(p+q+1)} D_{\tilde{\varepsilon}}(s) \tilde{\delta} \\
        &\qquad - (-1)^t\tilde{\varepsilon} y F(\alpha) - (-1)^{tp+tq+t+t'+q+t'q}I(\beta)x'\tilde{\delta} \\
        &\qquad - (-1)^{t+(t+t'+1)q}I(\beta)\tilde{\varepsilon} x - (-1)^{t+t'+tq}y'\tilde{\delta} F(\alpha) \\
        &= \tilde{\varepsilon} \tilde{\delta} F(s) - (-1)^{t(p+q-1)}\tilde{\varepsilon} G(s)\tilde{\delta} \\
        &\qquad + (-1)^{t(p+q+1)} \tilde{\varepsilon} G(s)\tilde{\delta} - (-1)^{(t+t')(p+q-1)}I(s)\tilde{\varepsilon}\tilde{\delta} \\
        &\qquad- (-1)^{t} \bigl(\tilde{\varepsilon} y +(-1)^{tq+t'} y'\tilde{\delta}\bigr) F(\alpha) \\
        &\qquad - (-1)^{t+(t+t'+1)q} I(\beta)\bigl( \tilde{\varepsilon} x + (-1)^{t'+tp} x' \tilde{\delta} \bigr) \\
        &= D_{\tilde{\varepsilon}\circ \tilde{\delta}}(s) - (-1)^{t+t'}\Bigl(((-1)^{t'}\tilde{\varepsilon} y + (-1)^{tq}y'\tilde{\delta})F(\alpha) \\
        &\qquad + (-1)^{q+(t+t')q}I(\beta)\bigl((-1)^{tp}x'\tilde{\delta}+ (-1)^{t'}\tilde{\varepsilon} x\bigr)  \Bigr).
    \end{align*}
    It remains to check that
    \begin{align*}
        d\big((-1)^{t'} \tilde{\varepsilon} x + (-1)^{tp} x'\tilde{\delta}\big) & = \tilde{\varepsilon} d(x) + (-1)^{tp} d(x')\tilde{\delta} \\
         & = \tilde{\varepsilon} D_{\tilde{\delta}}(\alpha) + (-1)^{tp} D_{\tilde{\varepsilon}}(\alpha) \tilde{\delta} \\
         & = \tilde{\varepsilon} \tilde{\delta} F(\alpha) - (-1)^{tp}\tilde{\varepsilon} G(\alpha)\tilde{\delta} \\
         &\qquad + (-1)^{tp}\tilde{\varepsilon} G(\alpha)\tilde{\delta} - (-1)^{(t+t')p}I(\alpha)\tilde{\varepsilon}\tilde{\delta} \\
         & = D_{\tilde{\varepsilon}\circ\tilde{\delta}}(\alpha),
    \end{align*}
    and similarly
    $$d\big((-1)^{t'} \tilde{\varepsilon} y + (-1)^{tq} y'\tilde{\delta}\big) = D_{\tilde{\varepsilon}\circ\tilde{\delta}}(\beta).$$
    Hence $f\in \langle b,a\rangle_{\varepsilon\circ\delta}$.

    If $\delta$ is a natural isomorphism, then $M_{\varepsilon}(b,a)\circ\delta_A=M_{\varepsilon\circ\delta}(b,a)$ by the naturality of $\delta$, while $\varepsilon_C\circ M_{\delta}(b,a)\subseteq M_{\varepsilon\circ\delta}(b,a)$ by the naturality of $\varepsilon$. Hence the left hand side is a coset of $M_{\varepsilon\circ\delta}(b,a)$ and the inclusion is an equality. The case of $\varepsilon$ is symmetric. If both are natural isomorphisms, each summand is a coset of $M_{\varepsilon\circ\delta}(b,a)$, so the equation holds in the quotient group by this subgroup.
\end{proof}

\begin{corollary} \label{inverse_product}
    Let $\delta\colon H^*(F)\Rightarrow H^*(G)$ be a graded natural isomorphism. Then, for any morphisms $a\colon A\to B$ and $b\colon B\to C$ in $H^*(\mathscr{A})$ such that $ba=0$, the naturality defect product $\langle b,a\rangle_{\delta}$ vanishes if and only if $\langle b,a\rangle_{\delta^{-1}}$ does.
\end{corollary}
\begin{proof}
    By part $(5)$ of Proposition \ref{product_props}, the equation
    $$ \delta_{C}^{-1} \circ \langle b,a \rangle_{\delta} \pm \langle b,a \rangle_{\delta^{-1}}\circ \delta_A = \langle b,a \rangle_{\id} $$
    holds in the quotient group by $M_{\id}(b,a)$. The right hand side vanishes by Corollary \ref{identity_product}, so one summand vanishes when the other does, and the result follows from the invertibility of $\delta_C$ and $\delta_A$.
\end{proof}

The naturality defect product does not only behave with respect to vertical composition but also with horizontal composition of graded natural transformations. We recall the definition in order to fix a sign convention.

\begin{definition}
    Let $F,F'\colon \mathscr{A} \to \mathscr{B}$, $G,G'\colon \mathscr{B}\to \mathscr{C}$ be dg functors and $\delta\colon F\Rightarrow F', \varepsilon\colon G\Rightarrow G'$ graded natural transformations of degree $t$ and $t'$ respectively. The \emph{horizontal composition} $\varepsilon*\delta\colon G\circ F \Rightarrow G'\circ F'$ is the graded natural transformation of degree $t+t'$ defined by
    $$ (\varepsilon*\delta)_X = \varepsilon_{F'(X)}\circ G(\delta_X) = (-1)^{tt'} G'(\delta_X)\circ \varepsilon_{F(X)}. $$
\end{definition}

\begin{proposition} \label{horizontal_props}
    Let $F,F'\colon \mathscr{A} \to \mathscr{B}$, $G,G'\colon \mathscr{B}\to \mathscr{C}$ be dg functors, $\delta\colon H^*(F)\Rightarrow H^*(F'), \varepsilon\colon H^*(G)\Rightarrow H^*(G')$ graded natural transformations and $a,b$ be morphisms in $H^*(\mathscr{A})$ such that their naturality defect products are defined. Then,
    $$\varepsilon_{F'(C)}\circ G(\langle b,a\rangle_{\delta} ) + (-1)^{|\delta| (|a|+|b|+1)} \langle F'(b),F'(a)\rangle_{\varepsilon}\circ G(\delta_A) \subseteq \langle b,a\rangle_{\varepsilon*\delta}, $$
    which is an equality if $\delta$ is a natural isomorphism. If, moreover, $\varepsilon$ is a natural isomorphism and $H^*(G)$ is full, the equation holds in the quotient group by $M_{\varepsilon*\delta}(b,a)$.
\end{proposition}

\begin{proof}
    Write $a=[\alpha]$ and $b=[\beta]$, and fix representatives $\tilde{\delta}_X$ of $\delta_X$ and $\tilde{\varepsilon}_X$ of $\varepsilon_X$ for every object $X$ in $\mathscr{A}$, and $s$ such that $d(s)=\beta\alpha$. Notice that 
    $$\varepsilon*\delta=\varepsilon_{F'}\circ G(\delta),$$
    where $G(\delta)\colon H^*(GF)\Rightarrow H^*(GF')$ and $\varepsilon_{F'}\colon H^*(GF')\Rightarrow H^*(G'F')$ are the graded natural transformations given by $G(\delta)_X=G(\delta_X)$ and $(\varepsilon_{F'})_X=\varepsilon_{F'(X)}$, with the obvious representatives. By part $(5)$ of Proposition \ref{product_props}, 
    $$ \varepsilon_{F'(C)}\circ\langle b,a\rangle_{G(\delta)} + (-1)^{|\delta| (|a|+|b|+1)} \langle b,a\rangle_{\varepsilon_{F'}}\circ G(\delta_A) \subseteq \langle b,a\rangle_{\varepsilon*\delta}. $$
    The desired inclusion then follows if $G(\langle b,a\rangle_{\delta})\subseteq \langle b,a\rangle_{G(\delta)}$ and $\langle F'(b),F'(a)\rangle_{\varepsilon}\subseteq \langle b,a\rangle_{\varepsilon_{F'}}$. Let us prove these claims. For the first one, let $x,y$ be such that $d(x)=D_{\tilde{\delta}}(\alpha)$ and $d(y)=D_{\tilde{\delta}}(\beta)$. Since $G$ is a graded functor,
    $$ G\bigl(m_{\tilde{\delta}}(\alpha,\beta,s,x,y)\bigr)=m_{G(\tilde{\delta})}(\alpha,\beta,s,G(x),G(y)), $$
    and, since it commutes with the differential, $d(G(x))=D_{G(\tilde{\delta})}(\alpha)$ and $d(G(y))=D_{G(\tilde{\delta})}(\beta)$. By Proposition \ref{all_elements_come_from_choices}, this shows that $G(\langle b,a\rangle_{\delta})\subseteq\langle b,a\rangle_{G(\delta)}$. For the second claim, notice that $D_{\tilde{\varepsilon}_{F'}}(\alpha)=D_{\tilde{\varepsilon}}(F'(\alpha))$, and similarly for $\beta$ and $s$. Hence, if $x',y'$ are such that $d(x')=D_{\tilde{\varepsilon}}(F'(\alpha))$ and $d(y')=D_{\tilde{\varepsilon}}(F'(\beta))$, then
    $$ m_{\tilde{\varepsilon}_{F'}}(\alpha,\beta,s,x',y')=m_{\tilde{\varepsilon}}(F'(\alpha),F'(\beta),F'(s),x',y'). $$
    By Proposition \ref{all_elements_come_from_choices} applied with $F'(s)$, every element of $\langle F'(b),F'(a)\rangle_{\varepsilon}$ is the class of the right hand side for some such $x',y'$, so it lies in $\langle b,a\rangle_{\varepsilon_{F'}}$. 

    The inclusion $\langle F'(b),F'(a)\rangle_{\varepsilon}\subseteq\langle b,a\rangle_{\varepsilon_{F'}}$ is always an equality, since both are cosets of the same subgroup. If $\delta$ is a natural isomorphism, so is $G(\delta)$, and part $(5)$ of Proposition \ref{product_props} says that the left hand side of the inclusion above equals $\langle b,a\rangle_{\varepsilon*\delta}$. Replacing $\langle b,a\rangle_{G(\delta)}$ by $G(\langle b,a\rangle_{\delta})$ does not change the sum: the difference lies in $\varepsilon_{F'(C)}\circ M_{G(\delta)}(b,a)\subseteq M_{\varepsilon*\delta}(b,a)$ by the naturality of $\varepsilon$, and, as in the proof of part $(5)$, the second summand is a coset of $M_{\varepsilon*\delta}(b,a)$. This gives the equality.

    If moreover $\varepsilon$ is a natural isomorphism and $H^*(G)$ is full, then $G(M_{\delta}(b,a))=M_{G(\delta)}(b,a)$, so $G(\langle b,a\rangle_{\delta})=\langle b,a\rangle_{G(\delta)}$, as an inclusion between cosets of the same subgroup. Hence the left hand side of the statement is that of part $(5)$ of Proposition \ref{product_props} for the natural isomorphisms $G(\delta)$ and $\varepsilon_{F'}$, and the equation holds in the quotient group by $M_{\varepsilon*\delta}(b,a)$.
\end{proof}

\section{On a theorem by Heller} \label{sec:heller}

The following setup follows Section 8 of \cite{amiot_2007_structure_triangulated_categories}. Let $\mathscr{T}$ be an additive category with split idempotents admitting a pretriangulated structure. We consider the abelianization of $\mathscr{T}$, i.e.\ the category $\Mod\mathscr{T}$ of finitely presented functors $F\colon \mathscr{T}^{op}\to \mathbf{Ab}$. Recall that, since $\mathscr{T}$ admits a pretriangulated structure, the category $\Mod\mathscr{T}$ is a Frobenius abelian category whose projective-injectives are the representable functors. This way, we can identify $\mathscr{T}$ with the full subcategory of $\Mod\mathscr{T}$ consisting of the projective objects via the Yoneda embedding:
$$\begin{array}{rcccl}
\mathscr{T} & \overset{\sim}{\longrightarrow} & \proj\mathscr{T} & \hookrightarrow & \Mod\mathscr{T}\\
X & \mapsto & \mathscr{T}(-,X) & = & \mathscr{T}(-,X) 
\end{array}$$

A complex over $\mathscr{T}$ is \emph{acyclic} if it becomes exact under this embedding. We further consider the category of acyclic complexes over $\Mod\mathscr{T}$ and its full subcategory of complexes with projective components
$$\Ac(\mathscr{T}) \overset{\sim}{\longrightarrow} \Ac(\proj\mathscr{T}) \hookrightarrow \Ac(\Mod\mathscr{T}).$$
The category $\Ac(\mathscr{T})$ is Frobenius and its projective-injectives are the contractible complexes. The 0-cocycles functor $Z^0\colon \Ac(\mathscr{T})\to \Mod\mathscr{T}$ sends projectives into projectives and induces an equivalence
$$Z^0\colon \underline{\Ac}(\mathscr{T})\to \underline{\Mod}(\mathscr{T})$$
which is a triangulated functor. Notice that $\underline{\Ac}(\mathscr{T})$ is both the stable category of $\Ac(\mathscr{T})$ and the homotopy category of acyclic complexes over $\mathscr{T}$. It has a natural triangulated structure, whose suspension functor is the usual shift $[1]$, defined on objects by $(X[1])^n := X^{n+1}$, $d^n_{X^{\bullet}[1]}:= -d^{n+1}_{X^{\bullet}}$.

Recall that a \emph{triangulated functor} $(F,\phi)\colon \mathscr{T}\to \mathscr{T}'$ between triangulated categories $\mathscr{T}$ and $\mathscr{T}'$ with suspension functors $\Sigma$ and $\Sigma'$ is a pair where $F\colon \mathscr{T}\to \mathscr{T}'$ is a functor and $\phi\colon F\circ \Sigma \cong \Sigma'\circ F$ is a natural isomorphism, such that $(F,\phi)$ takes distinguished triangles to distinguished triangles. In addition, a \emph{triangulated natural transformation} $\alpha\colon (F,\phi)\Rightarrow (F',\phi')$ between triangulated functors is a natural transformation $\alpha\colon F\Rightarrow F'$ such that the following square of natural transformations commutes:
$$\xymatrix{F \Sigma \ar@{=>}[r]^{\alpha \Sigma}\ar@{=>}[d]_{\phi} & F'\Sigma \ar@{=>}[d]^{\phi'}\\
\Sigma' F \ar@{=>}[r]_{\Sigma'\alpha} & \Sigma' F'} $$
One can check that the pair $([1],-\id)\colon \underline{\Ac}(\mathscr{T})\to \underline{\Ac}(\mathscr{T})$ is a triangulated functor. Now, let $\Sigma\colon \mathscr{T}\to \mathscr{T}$ be an autoequivalence and consider the induced functor $\Sigma\colon \underline{\Ac}(\mathscr{T})\to \underline{\Ac}(\mathscr{T})$ defined componentwise. Clearly, $(\Sigma,\id)\colon \underline{\Ac}(\mathscr{T})\to \underline{\Ac}(\mathscr{T})$ is also a triangulated functor.

A complex $X^{\bullet}$ over $\mathscr{T}$ is \emph{$3$-$\Sigma$-periodic} if $X^{n+3}=\Sigma X^n$ and $d^{n+3}_{X}=\Sigma d^n_{X}$ for every $n\in \mathbb{Z}$. A morphism $\kappa$ between $3$-$\Sigma$-periodic complexes is \emph{$\Sigma$-periodic} if $\kappa^{n+3}=\Sigma\kappa^n$ for every $n\in \mathbb{Z}$. A triangle
$$X\overset{\alpha}{\to} Y\overset{\beta}{\to} Z\overset{\gamma}{\to} \Sigma X$$
whose consecutive composites vanish gives rise to the $3$-$\Sigma$-periodic complex
$$\dots \to \Sigma^{-1}Z \overset{\Sigma^{-1} \gamma}{\longrightarrow} X \overset{\alpha}{\longrightarrow} Y \overset{\beta}{\longrightarrow} Z \overset{\gamma}{\longrightarrow} \Sigma X \overset{\Sigma \alpha}{\longrightarrow} \Sigma Y \to \dots$$
with $X$, $Y$ and $Z$ in degrees $-1$, $0$ and $1$. If the triangle is distinguished, this complex is acyclic, since the Yoneda embedding sends distinguished triangles to exact sequences.

We now state \cite[Thm.~16.4]{heller1968stable} in terms of $\underline{\Ac}(\mathscr{T})$ instead of the equivalent triangulated category $\underline{\Mod}(\mathscr{T})$.

\begin{theorem}[{\cite[Thm.~16.4]{heller1968stable}}] \label{heller_version}
    There is a bijection between the set of pretriangulated structures on $\mathscr{T}$ with suspension functor $\Sigma$ and the set of triangulated natural isomorphisms $(\Sigma,\id)\cong ([3],-\id)$ on $\underline{\Ac}(\mathscr{T})$. Moreover, the distinguished triangles in the pretriangulated structure associated to a natural isomorphism $\delta$ correspond to the $3$-$\Sigma$-periodic complexes $X^{\bullet}$ such that the morphism of complexes $\Sigma X^{\bullet}\to X^{\bullet}[3]$ given in degree $n$ by
    $$\tilde{\delta}_{X^{\bullet}}^n =(-1)^n \id_{\Sigma X^n}$$
    is a representative for $\delta_{X^{\bullet}}$.
\end{theorem}
\begin{proof}
    Let $\Delta$ be a pretriangulated structure on $\mathscr{T}$ with suspension functor $\Sigma$. We define the triangulated natural isomorphism $\delta=\delta(\Delta)\colon (\Sigma,\id)\cong ([3],-\id)$ as follows.
    Given a complex $X^{\bullet}$, since $\Delta$ is a pretriangulated structure, we can call $\alpha= d_{X^{\bullet}}^{-1}$ (the differential of $X^{\bullet}$ at degree $-1$) and complete  it into a distinguished triangle
    $$X^{-1} \overset{\alpha}{\to} X^0 \overset{\beta}{\to} Y \overset{\gamma}{\to} \Sigma X^{-1} $$
    which gives rise to a $3$-$\Sigma$-periodic acyclic complex $T(X^{\bullet})$.
    Notice that for every complex $X^{\bullet}$, $Z^0(X^{\bullet})=Z^0(T(X^{\bullet}))$ and so there is a homotopy equivalence $p\colon T(X^{\bullet})\to X^{\bullet}$ which induces the identity on $Z^0$. Similarly, since the functor $Z^3 = Z^0\circ [3]$ is also an equivalence, there is a homotopy equivalence $q\colon \Sigma X^{\bullet}[-3]\to T(X^{\bullet})$ which induces the identity on $Z^3(\Sigma X^{\bullet}[-3]) = Z^0(\Sigma X^{\bullet}) = Z^3(T(X^{\bullet}))$. Finally, we define
    $$\delta_{X^{\bullet}} :=  (p\circ q)[3] \in \Hom_{\underline{\Ac}(\mathscr{T})}(\Sigma X^{\bullet},X^{\bullet}[3]).$$
    The obtained natural isomorphism does not depend on the choice of triangle and so it only depends on the pretriangulated structure $\Delta$. This construction is equivalent to the one given in \cite{heller1968stable} and so it gives a bijection. In particular, the distinguished triangles in $\Delta$ are detected by $\delta$.

If $X^{\bullet}$ is a $3$-$\Sigma$-periodic complex which comes from a distinguished triangle,  then we can choose $T(X^{\bullet})$ to be $X^{\bullet}$. The following diagram shows a choice of representatives for $q[3]$ and $p[3]$ in this case:
$$  \xymatrix@C=1.4pc{ \Sigma X^{\bullet} \ar[d]^{q[3]} & \dots \ar[r] & X \ar[d]_{\id} \ar[r]^{\alpha} & Y \ar[d]_{-\id} \ar[r]^{\beta} & Z \ar[d]_{\id} \ar[r]^{\gamma} & \Sigma X \ar[d]_{-\id} \ar[r]^{\Sigma \alpha} & \Sigma Y \ar[r] \ar[d]_{\id} & \dots \\
     T(X^{\bullet})[3] \ar[d]^{p[3]} & \dots \ar[r] & X \ar@{=}[d] \ar[r]^{-\alpha} & Y \ar@{=}[d] \ar[r]^{-\beta} & Z \ar@{=}[d] \ar[r]^{-\gamma} & \Sigma X \ar@{=}[d] \ar[r]^{-\Sigma \alpha} & \Sigma Y \ar[r] \ar@{=}[d] & \dots \\
     X^{\bullet}[3] & \dots \ar[r] & X \ar[r]^{-\alpha} & Y \ar[r]^{-\beta} & Z \ar[r]^{-\gamma} & \Sigma X \ar[r]^{-\Sigma \alpha} & \Sigma Y \ar[r] & \dots  }$$
    With these choices, it is clear that $\delta_{X^{\bullet}}$ is of the stated form.
    \end{proof}

\begin{definition} \label{exact_comp}
    In the context of the theorem above, we say that an acyclic complex $X^{\bullet}$ is an \emph{exact triangle} with respect to the natural isomorphism $\delta$ if it is $3$-$\Sigma$-periodic and the morphism of complexes $\Sigma X^{\bullet}\to X^{\bullet}[3]$ given in degree $n$ by
     $$\tilde{\delta}_{X^{\bullet}}^n =(-1)^n \id_{\Sigma X^{n}}$$
     is a representative for $\delta_{X^{\bullet}}$.
\end{definition}

The second part of Theorem \ref{heller_version} states that, if a pretriangulated structure $\Delta$ is associated to a natural isomorphism $\delta$, then the distinguished triangles in $\Delta$ correspond to the $3$-$\Sigma$-periodic complexes which are exact triangles with respect to $\delta$ in the sense of the definition above. 

\begin{remark} \label{delta_shift}
    Recall that $\delta$ being a natural isomorphism of triangulated functors means that, for every acyclic complex $X^{\bullet}$,
$$ \delta_{X^{\bullet}}[1] = - \delta_{X^{\bullet}[1]}.$$
    In particular, if $X^{\bullet}$ is an exact triangle, then $X^{\bullet}[1]$ is also an exact triangle since, if $\tilde{\delta}_{X^{\bullet}}$ is a representative of $\delta_{X^{\bullet}}$ as in the definition above, then a representative for $\delta_{X^{\bullet}[1]}$ is given by
 $$\tilde{\delta}_{X^{\bullet}[1]}^n = - (\tilde{\delta}_{X^{\bullet}}[1])^n = - \tilde{\delta}_{X^{\bullet}}^{n+1} = - (-1)^{n+1} \id_{\Sigma X^{n+1}} = (-1)^n \id_{(\Sigma X^{\bullet}[1])^{n}}. $$
This fact is the translation of Verdier's axiom (TR2).
\end{remark}

\section{The octahedral axiom} \label{sec:octahedral}

We will now describe a particular class of instances of the naturality defect product in the context of pretriangulated categories. For that purpose, we describe an enhancement of the category $\underline{\Ac}(\mathscr{T})$. Let $\mathscr{T}$ be a pretriangulated category with split idempotents and suspension functor $\Sigma$. Recall the following definition:

\begin{definition}
    The \emph{dg category of cochain complexes} over $\mathscr{T}$, $\Ch(\mathscr{T})_{\mathrm{dg}}$, is defined to have as objects the cochain complexes over $\mathscr{T}$, as degree $i$ morphisms
    $$\hom^{i}(X^{\bullet},Y^{\bullet}) = \{ (f^j)_{j\in \mathbb{Z}} \; | \; f^j\in \mathscr{T}(X^j,Y^{j+i}) \}  $$
    and differential $\partial \colon \hom^{i}(X^{\bullet},Y^{\bullet}) \to \hom^{i+1}(X^{\bullet},Y^{\bullet}) $ given by
    $$\partial(f) := d_{Y^{\bullet}}\circ f - (-1)^{i} f\circ d_{X^{\bullet}}. $$
\end{definition}

The shift $[1]$ and the suspension $\Sigma$ extend to dg functors $\Ch(\mathscr{T})_{\mathrm{dg}}\to \Ch(\mathscr{T})_{\mathrm{dg}}$. On objects, $[1]$ is defined as in the previous section. On morphisms, it is given by
$$ (f[1])^j := (-1)^{i}f^{j+1} $$
for every $f\in \hom^i(X^{\bullet},Y^{\bullet})$ and every $j\in \mathbb{Z}$. We write $[m]$ for the $m$-fold composite of $[1]$, so that $(f[3])^j = (-1)^{i}f^{j+3}$. On the other hand, $\Sigma$ is extended componentwise, with no extra sign needed.

We will work in the full subcategory of $\Ch(\mathscr{T})_{\mathrm{dg}}$ formed by the acyclic complexes, which we denote by $\Ac(\mathscr{T})_{\mathrm{dg}}$. In this way, $Z^0(\Ac(\mathscr{T})_{\mathrm{dg}})=\Ac(\mathscr{T})$ and $H^0(\Ac(\mathscr{T})_{\mathrm{dg}})=\underline{\Ac}(\mathscr{T})$. The dg functors $[1]$ and $\Sigma$ restrict to dg functors $\Ac(\mathscr{T})_{\mathrm{dg}}\to \Ac(\mathscr{T})_{\mathrm{dg}}$, which induce on $\Ac(\mathscr{T})$ and on $\underline{\Ac}(\mathscr{T})$ the shift and the suspension of Section \ref{sec:heller}. 

Recall the following fact. If $\mathscr{A}$ is a pretriangulated dg category, as our dg categories of interest are, then $H^0(\mathscr{A})$ is triangulated and $H^*(\mathscr{A})$ is determined by it: the morphisms of degree $n$ from $X$ to $Y$ are the morphisms in $H^0(\mathscr{A})$ from $X$ to the $n$-th suspension of $Y$. Similarly, for dg functors $F,G\colon \mathscr{A}\to \mathscr{B}$, the graded natural transformations $H^*(F)\Rightarrow H^*(G)$ of degree $0$ are the same as the triangulated natural transformations $H^0(F)\Rightarrow H^0(G)$.

We consider the natural isomorphism $\delta\colon H^0(\Sigma)\Rightarrow H^0([3])$ given by the pretriangulated structure of $\mathscr{T}$ under the bijection in Theorem \ref{heller_version}. By the above, $\delta$ is a graded natural transformation $H^*(\Sigma)\Rightarrow H^*([3])$ of degree $0$, as required in Definition \ref{delta-product}.

Let $\kappa\colon A^{\bullet}\to B^{\bullet}$ be a morphism in $\Ac(\mathscr{T})$. Its \emph{mapping cone} is the complex $C^{\bullet}(\kappa)$ with $C^{n}(\kappa)=A^{n+1}\oplus B^{n}$ and differential
$$d_{C^{\bullet}(\kappa)} = \begin{pmatrix} -d_{A^{\bullet}} & 0 \\ \kappa & d_{B^{\bullet}} \end{pmatrix}.$$
We denote by $\iota\colon B^{\bullet}\to C^{\bullet}(\kappa)$ the inclusion in the cone and by $i$ its cohomology class. If $k\in \underline{\Ac}(\mathscr{T})(A^{\bullet},B^{\bullet})$ is the class of $\kappa$, then it is well known that the composition
$$A^\bullet \overset{k}{\longrightarrow} B^{\bullet} \overset{i}{\longrightarrow} C^{\bullet}(\kappa) $$
vanishes in $\underline{\Ac}(\mathscr{T})$. Therefore, it makes sense to define the naturality defect product $\langle i,k\rangle_{\delta}$ as in Definition \ref{delta-product}, which is an element of the quotient $$\frac{H^{-1}(\Ac(\mathscr{T})_{\mathrm{dg}})(\Sigma A^{\bullet},C^{\bullet}(\kappa)[3])}{i[3]\circ H^{-1}(\Ac(\mathscr{T})_{\mathrm{dg}})(\Sigma A^{\bullet},B^{\bullet}[3]) + H^{-1}(\Ac(\mathscr{T})_{\mathrm{dg}})(\Sigma B^{\bullet},C^{\bullet}(\kappa)[3]) \circ \Sigma k }.$$

\begin{proposition} \label{operation_to_delta}
    Let $A^{\bullet}$ and $B^{\bullet}$ be acyclic complexes, with fixed representatives $\tilde{\delta}_{A^{\bullet}}$ and $\tilde{\delta}_{B^{\bullet}}$ of $\delta$, and take $\tilde{\delta}_{A^{\bullet}[1]}:=-\tilde{\delta}_{A^{\bullet}}[1]$ as representative of $\delta_{A^{\bullet}[1]}$ (see Remark \ref{delta_shift}). Let $k=[\kappa]\in \underline{\Ac}(\mathscr{T})(A^{\bullet},B^{\bullet})$ and suppose that the composition
    $$A^\bullet \overset{k}{\longrightarrow} B^{\bullet} \overset{i}{\longrightarrow} C^{\bullet}(\kappa)$$
    satisfies $\langle i,k\rangle_{\delta}=0$. Then, there is some $x\in \hom^{-1}(\Sigma A^{\bullet},B^{\bullet}[3])$ satisfying $D_{\tilde{\delta}}(\kappa) = \partial(x)$ and such that the morphism of complexes $\Sigma C^{\bullet}(\kappa)\to C^{\bullet}(\kappa)[3]$ given in degree $n$ by
    \begin{equation} \label{forma}
        \begin{pmatrix} \tilde{\delta}_{A^{\bullet}[1]}^n & 0 \\
    x^{n+1} & \tilde{\delta}_{B^{\bullet}}^n
    \end{pmatrix}
    \end{equation}
    is a representative of $\delta_{C^{\bullet}(\kappa)}$.
\end{proposition}

\begin{proof} 
    Let
    $$\tilde{\delta}_{C^{\bullet}(\kappa)} = \begin{pmatrix} \tilde{\delta}_{11} & \tilde{\delta}_{12} \\
    \tilde{\delta}_{21} & \tilde{\delta}_{22}
    \end{pmatrix}\colon \Sigma A^{\bullet}[1]\oplus \Sigma B^{\bullet} \to A^{\bullet}[4]\oplus B^{\bullet}[3] $$
    be any representative of $\delta_{C^{\bullet}(\kappa)}$. Let us prove that, under our hypothesis, $\tilde{\delta}_{C^{\bullet}(\kappa)}$ is homotopic to a morphism of the form \eqref{forma}. Notice that a homotopy in this context is a morphism of degree $-1$
    $$h=\begin{pmatrix} h_{11} & h_{12} \\
    h_{21} & h_{22}
    \end{pmatrix}\in \hom^{-1}(\Sigma C^{\bullet}(\kappa),C^{\bullet}(\kappa)[3])$$ such that
    \begin{equation} \label{cone_homotopy}
        \partial(h) = \begin{pmatrix} \tilde{\delta}_{11} - \tilde{\delta}_{A^{\bullet}[1]} & \tilde{\delta}_{12} \\
        \tilde{\delta}_{21} + x[1] & \tilde{\delta}_{22} - \tilde{\delta}_{B^{\bullet}}
        \end{pmatrix},
    \end{equation}
    where the left hand side expands to
    $$
        \begin{pmatrix} \partial(h_{11})^n +h_{12}^{n+1}\Sigma\kappa^{n+1} & \partial(h_{12})^n \\
        \partial(h_{21})^n + h_{22}^{n+1}\Sigma\kappa^{n+1} - \kappa^{n+3} h_{11}^{n} & \partial(h_{22})^n -\kappa^{n+3} h_{12}^n
        \end{pmatrix}
    $$
    on the components.

    Now, suppose $\langle i,k\rangle_{\delta}$ vanishes. We keep our choices for the representatives of $\delta$ and our $\kappa$, and take
    $$\iota=\begin{pmatrix}
        0 \\ \id_{B^{\bullet}}
    \end{pmatrix} , \hspace{5mm} s=\begin{pmatrix}
        \id_{A^{\bullet}} \\ 0
    \end{pmatrix},$$
    so that $\iota\kappa = \partial(s)$. By Proposition \ref{all_elements_come_from_choices}, we can find $x$ and $y$ with
    $$D_{\tilde{\delta}}(\kappa) = \partial(x), \hspace{5mm} D_{\tilde{\delta}}(\iota) = \partial(y), $$
    such that
    \begin{equation} \label{vanishing}
        \big[\, D_{\tilde{\delta}}(s)-y(\Sigma\kappa) - (\iota[3]) x \,\big]=[0]\in H^{-1}(\Ac(\mathscr{T})_{\mathrm{dg}})(\Sigma A^{\bullet},C^{\bullet}(\kappa)[3]).
    \end{equation}
    One can check that the equation $D_{\tilde{\delta}}(\iota)=\partial(y)$ amounts to
    \begin{equation} \label{homotopy_y}
        \tilde{\delta}_{12}=\partial(y_1), \hspace{5mm} \tilde{\delta}_{22}-\tilde{\delta}_{B^{\bullet}}= \partial(y_2)-(\kappa[3]) y_1.
    \end{equation}
    Coming back to \eqref{vanishing}, we can find some $b\in \hom^{-2}(\Sigma A^{\bullet},C^{\bullet}(\kappa)[3])$ such that
    $$\begin{array}{ccl}
       0  & = & D_{\tilde{\delta}}\begin{pmatrix}
        \id_{A^{\bullet}} \\ 0
    \end{pmatrix} - \begin{pmatrix}
        y_1 \\ y_2
    \end{pmatrix}(\Sigma\kappa) - \begin{pmatrix}
        0 \\ \id_{B^{\bullet}[3]}
    \end{pmatrix} x - \partial(b) \\[4mm]
         & = & \begin{pmatrix} \tilde{\delta}_{11} & \tilde{\delta}_{12} \\
    \tilde{\delta}_{21} & \tilde{\delta}_{22}
    \end{pmatrix} \begin{pmatrix}
        \Sigma \id_{A^{\bullet}} \\ 0
    \end{pmatrix} - \begin{pmatrix}
        -\id_{A^{\bullet}[3]} \\ 0
    \end{pmatrix} \tilde{\delta}_{A^{\bullet}} \\[4mm]
         & & \qquad - \begin{pmatrix}
        y_1 \Sigma\kappa \\ y_2 \Sigma\kappa
    \end{pmatrix} - \begin{pmatrix}
        0 \\ x
    \end{pmatrix} - \begin{pmatrix}
        \partial(b_1) \\ \partial(b_2)-(\kappa[3]) b_1
    \end{pmatrix}.
    \end{array}  $$
    Evaluating at the $(n+1)$-th component, we get the equations
    \begin{equation}
        \tilde{\delta}_{11}^n+\tilde{\delta}_{A^{\bullet}}^{n+1} = y_1^{n+1}\Sigma\kappa^{n+1} + \partial(b_1)^{n+1}
    \end{equation}
    and
    \begin{equation}
        \tilde{\delta}_{21}^n-x^{n+1} = y_2^{n+1}\Sigma\kappa^{n+1} - \kappa^{n+3} b_1^{n+1} + \partial(b_2)^{n+1},
    \end{equation}
    which along with \eqref{homotopy_y} gives us a homotopy as in \eqref{cone_homotopy} taking
    $$h_{12}= y_1, \hspace{5mm} h_{22}= y_2, $$
    and defining, for every $n\in \mathbb{Z}$,
    $$h_{11}^n = b_1^{n+1}, \hspace{5mm} h_{21}^n = b_2^{n+1}.$$
\end{proof}

One of the most useful characterizations of the octahedral axiom is Neeman's axiom (TR4') in  \cite{neeman_2001_triangulated_categories}, which states that every commutative square between distinguished triangles

    $$\xymatrix{ X \ar[r]^{\alpha} \ar[d]_{\kappa^{-1}} & Y \ar[r]^{\beta} \ar[d]_{\kappa^0} & Z \ar[r]^{\gamma}  & \Sigma X \\ X' \ar[r]^{\alpha'} & Y' \ar[r]^{\beta'} & Z' \ar[r]^{\gamma'} & \Sigma X'  }$$
    can be completed into a triangle morphism $(\kappa^{-1},\kappa^0,\kappa^1)$ so that the cone triangle 
    $$\xymatrix{ Y \oplus X' \ar[rr]^{\begin{psmallmatrix}
    -\beta & 0 \\ \kappa^0 & \alpha'
\end{psmallmatrix}} & & Z\oplus Y' \ar[rr]^{\begin{psmallmatrix}
    -\gamma & 0 \\ \kappa^1 & \beta'
\end{psmallmatrix}} & & \Sigma X\oplus Z' \ar[rr]^{\begin{psmallmatrix}
    -\Sigma \alpha & 0 \\ \Sigma \kappa^{-1} & \gamma'
\end{psmallmatrix}} & & \Sigma Y \oplus \Sigma X' },$$
    is a distinguished triangle of the triangulated structure on $\mathscr{T}$. This axiom is equivalent to Verdier's axiom (TR4) by \cite[Theorem 1.8]{neeman_1991_new_axioms_triangulated} and \cite[Proposition 1.4.6]{neeman_2001_triangulated_categories}. 
    
    The above triangle can be completed into a $\Sigma$-periodic complex, which is precisely the mapping cone of the corresponding $\Sigma$-periodic morphism $\kappa$ between the complexes associated to each distinguished triangle.
    By our version of Heller's theorem, the cone triangle of $(\kappa^{-1},\kappa^0,\kappa^1)$ is distinguished if, and only if, $C^{\bullet}(\kappa)$ is an exact triangle with respect to the natural isomorphism $\delta$ associated with the pretriangulated structure of $\mathscr{T}$, in the sense of Definition \ref{exact_comp}. In this scenario, the mapping cone is always acyclic since every map between acyclic complexes is trivially a quasi-isomorphism. Furthermore, $C^{\bullet}(\kappa) $ is clearly $3$-$\Sigma$-periodic, so it is a \emph{candidate triangle}.

    Notice that not every completion of the initial square gives an exact cone; it is enough that one of them does. We therefore have freedom to choose the third component of the morphism in a good way.

\begin{proposition} \label{change_of_cone}
    Suppose $k\in \underline{\Ac}(\mathscr{T})(A^{\bullet},B^{\bullet})$ is a morphism between exact triangles with a $\Sigma$-periodic representative $\kappa= (\kappa^{-1},\kappa^0,\kappa^1) $ such that the composition
    $$A^\bullet \overset{k}{\longrightarrow} B^{\bullet} \overset{i}{\longrightarrow} C^{\bullet}(\kappa)$$
    satisfies $\langle i,k\rangle_{\delta}=0$. Then, $k$ has another $\Sigma$-periodic representative $\psi = (\kappa^{-1},\kappa^0,\psi^1) $ such that $C^{\bullet}(\psi)$ is an exact triangle.
\end{proposition}

\begin{proof}
    Observe that every other representative of $k$ is of the form $\psi=\kappa-\partial(\eta)$, for some $\eta\in \hom^{-1}(A^{\bullet},B^{\bullet}) $. The complexes $C^{\bullet}(\kappa)$ and $C^{\bullet}(\psi)$ are isomorphic by the isomorphism $\phi$ given by
    $$\phi^n = \begin{pmatrix}
        \id_{A^{n+1}} & 0 \\
        \eta^{n+1} & \id_{B^{n}}
    \end{pmatrix}.$$ 
    Since $A^{\bullet}$ and $B^{\bullet}$ are exact triangles, we can take the representatives $\tilde{\delta}_{A^{\bullet}}^n=(-1)^n\id_{\Sigma A^n}$ and $\tilde{\delta}_{B^{\bullet}}^n=(-1)^n\id_{\Sigma B^n}$, so that $\tilde{\delta}_{A^{\bullet}[1]}^n=(-1)^n\id_{\Sigma A^{n+1}}$. Let $x$ be as in the conclusion of Proposition \ref{operation_to_delta} applied to $\kappa$ and these representatives. We choose $\eta$ so that
    \begin{equation} \label{recurrence}
    \eta^{n+4} = \Sigma \eta^{n+1} - (-1)^n x^{n+1} \hspace{5mm} \text{for all } n\in \mathbb{Z},
    \end{equation}
    which we can construct inductively. Specifically, we take the initial values $\eta^i = 0$ for $i\in \{-1,0,1\}$, which gives $\kappa^{-1}=\psi^{-1}$ and $\kappa^0=\psi^0$. We will now prove that, with any choice satisfying the recurrence relation, $C^{\bullet}(\psi)$ is an exact triangle.
    By naturality of $\delta$, there is a representative $\tilde{\delta}'$ of $\delta_{C^{\bullet}(\psi)}$ such that the square
    $$\xymatrix{ \Sigma C^{\bullet}(\kappa) \ar[rrr]^{\begin{pmatrix}
        \id_{\Sigma A^{n+1}} & 0 \\
        \Sigma\eta^{n+1} & \id_{ \Sigma B^{n}}
    \end{pmatrix}} \ar[d]_{\begin{pmatrix}
        \tilde{\delta}_{A^{\bullet}[1]}^n & 0 \\
        x^{n+1} & \tilde{\delta}_{B^{\bullet}}^n
    \end{pmatrix}} & & & \Sigma C^{\bullet}(\psi) \ar[d]^{(\tilde{\delta}')^n} \\
    C^{\bullet}(\kappa)[3] \ar[rrr]_{\begin{pmatrix}
        \id_{A^{n+4}} & 0 \\
        \eta^{n+4} & \id_{B^{n+3}}
    \end{pmatrix}} & & & C^{\bullet}(\psi)[3]
    }$$
    commutes for all $n \in \mathbb{Z}$, where we get the matrix on the left as a representative for $\delta_{C^{\bullet}(\kappa)}$ by Proposition \ref{operation_to_delta}. It follows by our choice of $\eta$ that $\tilde{\delta}'$ is of the form
    $$ (\tilde{\delta}')^n =  \begin{pmatrix}
        (-1)^n \id_{\Sigma A^{n+1}} & 0 \\
        0 & (-1)^n \id_{\Sigma B^n} \end{pmatrix} = (-1)^n \id_{\Sigma C^n(\psi)} , $$
    and so $C^{\bullet}(\psi)$ is an exact triangle.

    We now use that $\kappa$ is a $\Sigma$-periodic morphism between exact triangles.
    This immediately gives the vanishing of $D_{\tilde{\delta}}(\kappa)$ and so $x$ is a cocycle. We get
    \begin{align*}
        \partial(\eta)^{n+3} & = d_{B^{\bullet}}^{n+2} \eta^{n+3} + \eta^{n+4} d_{A^{\bullet}}^{n+3} \\ 
        & = d_{B^{\bullet}}^{n+2} ( \Sigma \eta^{n} - (-1)^{n-1} x^n ) + ( \Sigma \eta^{n+1} - (-1)^n x^{n+1}) d_{A^{\bullet}}^{n+3} \\
        & = d_{\Sigma B^{\bullet}}^{n-1} \Sigma \eta^{n} + \Sigma \eta^{n+1} d_{\Sigma A^{\bullet}}^{n} + (-1)^{n-1} \big( d_{B^{\bullet}[3]}^{n-1} x^n + x^{n+1} d_{\Sigma A^{\bullet}}^n \big) \\
        & = \partial(\Sigma \eta)^n + (-1)^{n-1} \partial(x)^n \\
        & = \Sigma \partial(\eta)^n.
    \end{align*}
    It follows that $\partial(\eta)$ and thus $\psi$ are also $\Sigma$-periodic as stated.
\end{proof}

We are now ready to state and prove our main theorem.

\begin{theorem} \label{main_theorem}
    Let $\mathscr{T}$ be as above. The following are equivalent:
    \begin{enumerate}
        \item $\mathscr{T}$ is triangulated.
        \item For every pair $A^{\bullet}, B^{\bullet}$ of exact triangles and for every $k\in \underline{\Ac}(\mathscr{T})(A^{\bullet},B^{\bullet})$, the product $\langle i,k\rangle_{\delta}$ vanishes for some $\Sigma$-periodic representative of $k$, where $i$ denotes the class of the inclusion of $B^{\bullet}$ into the corresponding mapping cone.
        \item For every pair  $A^{\bullet}, B^{\bullet}$ of exact triangles and for every $k\in \underline{\Ac}(\mathscr{T})(A^{\bullet},B^{\bullet})$, the product $\langle g,k\rangle_{\delta}$ vanishes for any morphism $g$ in $\underline{\Ac}(\mathscr{T})$ such that the naturality defect product is defined.
        \item The naturality defect product $\langle g,f\rangle_{\delta}$ vanishes for every pair of morphisms $f,g$ in $\underline{\Ac}(\mathscr{T})$ such that $gf=0$.
    \end{enumerate}
\end{theorem}

\begin{proof}
    We start by proving that $(2)$ implies the more general $(3)$. We recall that a $\Sigma$-periodic representative $\kappa$ for $k$ always exists since $\mathscr{T}$ is pretriangulated. Moreover, using the canonical triangulated structure of $\underline{\Ac}(\mathscr{T})$, for any $g$ such that $gk=0$ we obtain a factorization
    $$\xymatrix{ A^{\bullet} \ar[r]^k \ar[dr]_0 & B^{\bullet} \ar[r]^i \ar[d]^g & C^{\bullet}(\kappa) \ar@{.>}[dl]^h \\ & D^{\bullet} &  }.$$
    By hypothesis, $0\in \langle i,k\rangle_{\delta} $, and so, using part $(3)$ of Proposition \ref{product_props},
    $$0\in h[3] \circ \langle i,k\rangle_{\delta} \subseteq \langle g,k\rangle_{\delta}.$$

    Note that, as a consequence of the last argument, the vanishing of $\langle i,k\rangle_{\delta}$ is actually independent of the choice of representative $\kappa$ of $k$.

    To prove that $(3)$ implies the more general $(4)$, we recall that every acyclic complex is homotopy equivalent to an exact triangle by (TR1). Therefore, we see that the hypothesis of $A^{\bullet}$ and $B^{\bullet}$ being exact triangles is superfluous by parts $(3)$ and $(4)$ of Proposition \ref{product_props} respectively.

    Finally, we prove $(1)\Leftrightarrow (2)$ using Neeman's characterization and Theorem \ref{heller_version}.
    For the `only if' part, notice that the octahedral axiom implies that there is a $\Sigma$-periodic representative $\kappa$ of $k$ such that $C^{\bullet}(\kappa)$ is an exact triangle. Periodicity of the representatives $\kappa$ and $\iota$ and the fact that all $A^{\bullet}, B^{\bullet}$ and $C^{\bullet}(\kappa)$ are exact triangles imply that we can take $x=0$ and $y=0$ when defining $\langle i,k\rangle_{\delta}$ as in \eqref{vanishing}. One can check that, taking $s$ as in the proof of Proposition \ref{operation_to_delta}, $D_{\tilde{\delta}}(s)=0$ and so $m_{\tilde{\delta}}(\kappa,\iota,s,0,0)$ is zero.

    For the `if' part, consider a square
    $$\xymatrix{ X \ar[r]^{\alpha} \ar[d]_{\kappa^{-1} } & Y \ar[r]^{\beta} \ar[d]_{\kappa^0} & Z \ar[r]^{\gamma}  & \Sigma X \\ X' \ar[r]^{\alpha'} & Y' \ar[r]^{\beta'} & Z' \ar[r]^{\gamma'} & \Sigma X'  }$$
    between distinguished triangles. By (TR3), there is a completion into a triangle morphism $(\kappa^{-1},\kappa^0,\kappa^1)$. Let $\kappa$ be $(\kappa^{-1},\kappa^0,\kappa^1)$ seen as a $\Sigma$-periodic morphism between the associated complexes and $k$ its cohomology class. Since the vanishing of $\langle i,k\rangle_{\delta}$ is independent of the choice of representative, it must vanish for this $\kappa$.
    Then, Proposition \ref{change_of_cone} gives us another periodic complex morphism $\psi = (\kappa^{-1},\kappa^0,\psi^1)$ with exact cone, which means that the cone triangle of the completion $(\kappa^{-1},\kappa^0,\psi^1)$ in the original diagram is a distinguished triangle. This proves Neeman's axiom.
\end{proof}

Accordingly, we call a non-vanishing product $\langle g,f\rangle_{\delta}$ as in the theorem above a \emph{local obstruction to} (TR4).

\section{Examples} \label{sec:examples_section}

In this section, we will give some examples of non-vanishing obstructions, following constructions in the recent literature. The first part will be devoted to the examples given in \cite{Che26v2}, while the second part will focus on the explicit example in \cite{Name_Name_2026}, giving an obstruction which does not arise as in the first part; see Remark \ref{not_gemstone}. The second part yields a proof by hand that the example of \cite{Name_Name_2026} is not triangulated, as opposed to the original proof, which is computer-assisted.

\subsection{Gemstone-generated obstructions}

We start by adapting a definition from \cite{Che26v2} to the case we are interested in. Let $\Bbbk$ be a field and $\mathscr{T}$ a Hom-finite $\Bbbk$-linear triangulated category with split idempotents.

\begin{definition}[{\cite[Def.~5.3]{Che26v2}}]\label{def:gemstone}
An indecomposable object $X\in\underline{\Mod}(\mathscr{T})$ is a \emph{gemstone} if:
\begin{enumerate}
    \item[(G1)] there is an isomorphism $\gamma\colon X\xrightarrow{\sim}\Omega^{-2}X$, and $\underline{\Hom}_{\mathscr{T}}(X,\Omega^{-1}X)=0$;
    \item[(G2)] there is an almost-vanishing morphism $\alpha\colon X\to X$ \cite[\S 5]{Che26v2} which is not an isomorphism, lies in the center of $\underline{\End}_{\mathscr{T}}(X)$ and satisfies $\gamma\circ\alpha=\Omega^{-2}\alpha\circ\gamma$;
    \item[(G3)] in the almost-split triangle
    $$X\longrightarrow U\longrightarrow\Omega^{-1}X\xrightarrow{\Omega^{-1}\alpha}\Omega^{-1}X,$$
    $U$ is not isomorphic to $\Omega^{-1}U$.
\end{enumerate}
\end{definition}

The main theorem of \cite{Che26v2} states that if $\underline{\Mod}(\mathscr{T})$ contains a gemstone, then $\mathscr{T}$ admits a pretriangulated structure $\Delta$ with the same suspension as in the given triangulated structure, which is not triangulated. The Heller parameter of this pretriangulated structure (see Remark \ref{direction}) is of the form
$$ \Theta = \Theta_0 \circ \Sigma u \colon \Sigma \Rightarrow \Omega^{-3}, $$
where $\Theta_0$ is the Heller parameter associated to the triangulated structure $\Delta_0$ on $\mathscr{T}$ and $u$ is a triangulated natural automorphism of the identity functor which is the identity on every indecomposable not isomorphic to the gemstone $X$ or to $\Omega^{-1}X$.

\begin{remark} \label{direction}
    In its original form \cite[Thm.~16.4]{heller1968stable}, Heller's theorem records a pretriangulated structure $\Delta$ by a natural isomorphism
    $$ \Theta \colon \Sigma \overset{\sim}{\Longrightarrow} \Omega^{-3} $$
on $\underline{\Mod}(\mathscr{T})$. This is what \cite{Che26v2} calls a \emph{Heller parameter}. The equivalence $Z^0\colon \underline{\Ac}(\mathscr{T})\to \underline{\Mod}(\mathscr{T})$ of Section \ref{sec:heller} satisfies $Z^0(X^{\bullet}[1])=\Omega^{-1}Z^0(X^{\bullet})$, and it transports the natural isomorphism $\Sigma\Rightarrow [3]$ associated to $\Delta$ by Theorem \ref{heller_version} to $\Theta$.
\end{remark}

We denote by $\theta_0$ and $\theta$ the natural isomorphisms associated to $\Delta_0$ and $\Delta$ by Theorem \ref{heller_version}.

The authors of \cite{Che26v2} use the following notion throughout their paper.

\begin{definition}[{\cite[Def.~3.5]{christensen_frankland_2022_good_morphisms}}] \label{def:lightning_flash}
    Let $\xi\colon X\xrightarrow{u}Y\xrightarrow{v}Z\xrightarrow{w}\Sigma X $
    be a distinguished triangle in a pretriangulated category. A morphism $\ell\colon Z\to Z$ is a \emph{lightning flash} on $\xi$ if $\ell=v\circ\varphi\circ w$ for some $\varphi\colon\Sigma X\to Y$.
\end{definition}

There is a relation between this concept and the vanishing of some naturality defect products.

\begin{proposition} \label{lightning_flash}
    Let $\eta\colon \id_{\underline{\Ac}(\mathscr{T})} \Rightarrow \id_{\underline{\Ac}(\mathscr{T})}$ be a triangulated natural transformation and let
    $$\xi\colon \hspace{4mm} A^{\bullet} \overset{k}{\longrightarrow} B^{\bullet} \overset{g}{\longrightarrow} D^{\bullet} \overset{h}{\longrightarrow} A^{\bullet}[1]$$
    be a distinguished triangle in $\underline{\Ac}(\mathscr{T})$. Assume that $\eta_{A^{\bullet}}= 0$ and $\eta_{B^{\bullet}}=0$. Then, the product $\langle g,k\rangle_{\eta}$ vanishes if and only if $\eta_{D^{\bullet}}$ is a lightning flash on $\xi$.
\end{proposition}

\begin{proof}
    We first prove the lemma for a standard triangle. Let $\kappa$ be a representative of $k$ and consider
    $$\xi'\colon \hspace{4mm} A^{\bullet} \overset{k}{\longrightarrow} B^{\bullet} \overset{i}{\longrightarrow} C^{\bullet}(\kappa) \overset{p}{\longrightarrow} A^{\bullet}[1]$$
    in $\underline{\Ac}(\mathscr{T})$, where $i$ and $p$ denote the inclusion into and the projection from the mapping cone.

    \emph{Claim:} there exists a morphism $m\colon A^{\bullet}[1]\to C^{\bullet}(\kappa)$ which satisfies
    $$ \eta_{C^{\bullet}(\kappa)} = m p $$
    and whose identification with a degree $-1$ morphism $A^{\bullet}\to C^{\bullet}(\kappa)$ represents $\langle i, k\rangle_{\eta}$.

    Write $C^{\bullet}= C^{\bullet}(\kappa)$. We construct $m$ as in Definition \ref{delta-product}. Take the obvious representative $\iota$ of $i$ and $s = \begin{psmallmatrix} 1 \\ 0 \end{psmallmatrix}$ with $\partial(s) = \iota\kappa$. Take $x=0$, which we can do since $\eta$ vanishes at $A^{\bullet}$ and $B^{\bullet}$, and take $y$ such that $\partial(y) = D_{\tilde{\eta}}(\iota) = \tilde{\eta}_{C^{\bullet}}\iota $. Then, $m$ is the class of
    $$ \tilde{\eta}_{C^{\bullet}}s - y\kappa. $$
    Consider the morphism $r = \begin{psmallmatrix} 0 & 1 \end{psmallmatrix}\colon C^{\bullet}\to B^{\bullet}$ and let $\rho=\begin{psmallmatrix} 1 & 0 \end{psmallmatrix}$ be the obvious representative of $p$. One then gets
    $$ s\rho + \iota r = \id_{C^{\bullet}} \; \; \text{and} \; \; \partial(r)= -\kappa \rho.$$
    It follows that
    $$ \partial(yr) = \tilde{\eta}_{C^{\bullet}}\iota r + y\kappa \rho $$
    and so
    $$ \tilde{\eta}_{C^{\bullet}} - \big(\tilde{\eta}_{C^{\bullet}}s - y\kappa\big)\rho = \partial(yr) $$
    giving the claim.

    We now use the claim to prove the lemma for $\xi'$. By Definition \ref{def:lightning_flash}, $\eta_{C^{\bullet}}$ is a lightning flash on $\xi'$ if and only if $\eta_{C^{\bullet}}=i\varphi p$ for some $\varphi\colon A^{\bullet}[1]\to B^{\bullet}$, while $\langle i,k\rangle_{\eta}$ vanishes if and only if $m = i\varphi + \psi k[1]$ for some such $\varphi$ and some $\psi\colon B^{\bullet}[1]\to C^{\bullet}$. In the latter case, $\eta_{C^{\bullet}} = mp = i\varphi p$, since $k[1]\circ p=0$. Conversely, if $\eta_{C^{\bullet}} = i\varphi p$, then $\big(m-i\varphi\big)p=0$ and, rotating $\xi'$, we obtain a factorization
    $$\xymatrix@C=5pc@R=3pc{ C^{\bullet} \ar[r]^{p} \ar[dr]_0 & A^{\bullet}[1] \ar[r]^{-k[1]} \ar[d]^(.4){m-i\varphi} & B^{\bullet}[1] \ar@{.>}[dl]^{\psi} \\ & C^{\bullet} & }$$
    so that $m = i\varphi - \psi k[1]$.

    Finally, for a general distinguished triangle $\xi$, there is an isomorphism of triangles $(\id_{A^{\bullet}},\id_{B^{\bullet}},c)$ from $\xi'$ to $\xi$, so that $ci=g$ and $hc=p$. By part $(3)$ of Proposition \ref{product_props}, $\langle g,k\rangle_{\eta}$ vanishes if and only if $\langle i,k\rangle_{\eta}$ does. On the other hand, naturality gives that if $\eta_{C^{\bullet}} = i\varphi p$, then
    $$ \eta_{D^{\bullet}} = c\, i\varphi p\, c^{-1}=g\varphi h,$$
    so $\eta_{D^{\bullet}}$ is a lightning flash on $\xi$ if and only if $\eta_{C^{\bullet}}$ is one on $\xi'$.
\end{proof}

Observe that \cite[Prop.~4.2]{Che26v2} is a corollary of this lemma and Theorem \ref{main_theorem}. We now give the non-vanishing obstruction. Following \cite{Che26v2}, a gemstone $X$ has an associated distinguished triangle
$$ \Omega^{-1} X \xrightarrow{s} U \xrightarrow{w} Y \xrightarrow{z} \Omega^{-2}X, $$
where $U$ is as in (G3). Take acyclic complexes $P^{\bullet}$, $Q^{\bullet}$ and $R^{\bullet}$ with $Z^0(P^{\bullet})=U$, $Z^0(Q^{\bullet})=Y$ and $Z^0(R^{\bullet})=\Omega^{-2}X$, and let $k\colon P^{\bullet}\to Q^{\bullet}$ and $g\colon Q^{\bullet}\to R^{\bullet}$ be the morphisms in $\underline{\Ac}(\mathscr{T})$ corresponding to $w$ and $z$.
\begin{proposition} \label{gemstone_obstruction}
    Let $X$ be a gemstone and let $\theta$, $k$ and $g$ be as above. Then, the naturality defect product $\langle g,k \rangle_{\theta}$ does not vanish.
\end{proposition}

\begin{proof}
    Under $Z^0$, the decomposition $\Theta=\Theta_0\circ\Sigma u$ becomes $\theta=\theta_0\circ\Sigma u$, where we also denote by $u$ the corresponding triangulated natural automorphism of $\id_{\underline{\Ac}(\mathscr{T})}$. Since $\Delta_0$ is triangulated, $\langle g,k\rangle_{\theta_0}$ vanishes by Theorem \ref{main_theorem}. Hence, by part $(5)$ of Proposition \ref{product_props}, $\langle g,k\rangle_{\theta}$ vanishes if and only if $\langle g,k\rangle_{\Sigma u}$ does, which, by Proposition \ref{horizontal_props} applied to $\Sigma u=\id_{\Sigma}*u$, happens if and only if $\langle g,k\rangle_{u}$ vanishes. Defining $\eta= u - \id $, this is equivalent to the vanishing of $\langle g,k\rangle_{\eta}$, by part $(2)$ of Proposition \ref{product_props} and Corollary \ref{identity_product}.

    By \cite[Lem.~5.5 and Lem.~5.7]{Che26v2}, neither $X$ nor $\Omega^{-1}X$ is a direct summand of $U$ or $Y$, and so $\eta$ vanishes at $P^{\bullet}$ and $Q^{\bullet}$. Since $Z^0$ is a triangulated equivalence, the rotated triangle
    $$ \chi\colon \hspace{4mm} U \xrightarrow{w} Y \xrightarrow{z} \Omega^{-2}X \xrightarrow{-\Omega^{-1}s} \Omega^{-1}U $$
    lifts to a distinguished triangle in $\underline{\Ac}(\mathscr{T})$ of the form
    $$ \xi\colon \hspace{4mm} P^{\bullet} \overset{k}{\longrightarrow} Q^{\bullet} \overset{g}{\longrightarrow} R^{\bullet} \overset{h}{\longrightarrow} P^{\bullet}[1]. $$
    By Proposition \ref{lightning_flash}, it suffices to prove that $\eta_{R^{\bullet}}$ is not a lightning flash on $\xi$. Since triangulated functors preserve lightning flashes, it is enough to show that $Z^0(\eta_{R^{\bullet}})=\eta_{\Omega^{-2}X}$ is not a lightning flash on $\chi$. But $\eta_{\Omega^{-2}X}\neq 0$ \cite[proof of Thm.~6.2]{Che26v2}, so this follows from \cite[Prop.~5.8]{Che26v2}.
\end{proof}

As a consequence, we recover \cite[Thm.~1.1]{Che26v2}. Concrete examples of gemstones are given in \cite[\S 7]{Che26v2}.

\subsection{An explicit obstruction over \texorpdfstring{$\Pi(\mathbb{D}_4)$}{P(D4)}} \label{sec:pi_d4}

We now follow the example in \cite{Name_Name_2026}. Let $\Bbbk = \overline{\mathbb{F}}_2$ be the algebraic closure of the field with two elements and
$$ \Lambda = \Pi(\mathbb{D}_4) $$
the preprojective algebra of type $\mathbb{D}_4$ over $\Bbbk$. Explicitly, we can write $\Lambda$ as the quotient of the path algebra $\Bbbk Q$ of the quiver
$$ Q = \vcenter{\xymatrix@R=2pc@C=1.2pc{
 & & 1 \ar@/^/[d]^{b_1} & & \\
3 \ar@/^/[rr]^{b_3} & & 0 \ar@/^/[u]^{a_1} \ar@/^/[ll]^{a_3} \ar@/^/[rr]^{a_2} & & 2 \ar@/^/[ll]^{b_2}
}}$$
by the ideal generated by the mesh relations
$$ b_1 a_1 , \hspace{5mm} b_2 a_2 , \hspace{5mm} b_3 a_3, \hspace{5mm} a_1 b_1 + a_2 b_2 + a_3 b_3 . $$
Here paths are concatenated from left to right. No signs are needed in the last relation since $\operatorname{char}\Bbbk = 2$.

The algebra $\Lambda$ has dimension $28$ over $\Bbbk$. Since $\operatorname{char}\Bbbk = 2$, it is symmetric and its bimodule periodicity is untwisted \cite[\S 4]{brenner_butler_king_2002}, so the Białkowski--Erdmann--Skowro\'nski bimodule extension \cite[Thm.~2.2]{bialkowski_erdmann_skowronski_2007} takes the form
\begin{equation} \label{bimodule_resolution}
    \mathcal{E}\colon \hspace{4mm} 0 \longrightarrow \Lambda \longrightarrow B_2 \longrightarrow B_1 \longrightarrow B_0 \longrightarrow \Lambda \longrightarrow 0
\end{equation}
whose middle terms are projective bimodules; it realizes the isomorphism $\Omega^3_{\Lambda^{e}}(\Lambda)\cong \Lambda$, where $\Lambda^{e}=\Lambda\otimes_{\Bbbk} \Lambda^{op}$.
We now take $\mathscr{T}=\proj \Lambda$, the category of finitely generated projective right $\Lambda$-modules, so that $\Mod\mathscr{T}\simeq \Mod\Lambda$.
Tensoring with \eqref{bimodule_resolution} gives an exact sequence of exact functors
$$ 0 \longrightarrow \id \longrightarrow -\otimes_{\Lambda}B_2 \longrightarrow -\otimes_{\Lambda}B_1 \longrightarrow -\otimes_{\Lambda}B_0 \longrightarrow \id \longrightarrow 0 $$
on $\Mod\Lambda$ whose middle terms take values in $\proj\Lambda$. This is precisely the datum required by \cite[Thm.~8.1]{amiot_2007_structure_triangulated_categories}, which therefore endows $\mathscr{T}$ with a triangulated structure $\Delta_0$ whose suspension functor is $\Sigma=\id$.

Let $\Theta_0\colon \id \Rightarrow \Omega^{-3}$ be the Heller parameter of $\Delta_0$ and $\theta_0\colon \id \Rightarrow [3]$ the natural isomorphism associated to $\Delta_0$ by Theorem \ref{heller_version}, so that $Z^0(\theta_0)=\Theta_0$; see Remark \ref{direction}.

We now introduce the modules on which the example rests. Let $\alpha\in \Bbbk$ be such that $\alpha^2+\alpha+1=0$. For $n=1,2,3$, the \emph{band module} $X_n$ is the representation displayed below, whose matrices along $a_i$ and $b_i$ we denote by $A_{i,n}$ and $B_{i,n}$. They act on column vectors, so that, for example, the path $b_ia_i$ acts as $A_{i,n}B_{i,n}$. The first one is
$$ X_1 = \vcenter{\xymatrix@R=2.4pc@C=2.6pc{
 & & \Bbbk \ar@/^/[d]^{\begin{psmallmatrix} 0 \\ 1\end{psmallmatrix}} & & \\
\Bbbk \ar@/^/[rr]^{\begin{psmallmatrix} 0 \\ 1+\alpha\end{psmallmatrix}} & & \Bbbk^2 \ar@/^/[u]^{(1\; 0)} \ar@/^/[ll]^{(1\; 0)} \ar@/^/[rr]^{(1\; 0)} & & \Bbbk \ar@/^/[ll]^{\begin{psmallmatrix} 0 \\ \alpha \end{psmallmatrix}}
}},$$
the second one
$$ X_2 = \vcenter{\xymatrix@R=2.4pc@C=2.6pc{
 & & \Bbbk^2 \ar@/^/[d]^{B_{1,2}} & & \\
\Bbbk^2 \ar@/^/[rr]^{B_{3,2}} & & \Bbbk^4 \ar@/^/[u]^{A_{1,2}} \ar@/^/[ll]^{A_{3,2}} \ar@/^/[rr]^{A_{2,2}} & & \Bbbk^2 \ar@/^/[ll]^{B_{2,2}}
}}$$
where
\begin{gather*}
    A_{1,2} = A_{2,2} = A_{3,2} = \left(\begin{array}{@{\,}c|c@{\,}} I_2 & O_2 \end{array}\right), \\[2ex]
    B_{1,2} = \left(\begin{array}{@{\,}c@{\,}} O_2 \\ \hline I_2 \end{array}\right), \qquad
    B_{2,2} = \left(\begin{array}{@{\,}cc@{\,}} \multicolumn{2}{c}{\raisebox{-.5\normalbaselineskip}[0pt][0pt]{$O_2$}} \\ & \\ \hline \alpha&1 \\ 0&\alpha \end{array}\right), \qquad
    B_{3,2} = \left(\begin{array}{@{\,}cc@{\,}} \multicolumn{2}{c}{\raisebox{-.5\normalbaselineskip}[0pt][0pt]{$O_2$}} \\ & \\ \hline 1+\alpha&1 \\ 0&1+\alpha \end{array}\right)
\end{gather*}
and the third one
$$ X_3 = \vcenter{\xymatrix@R=2.4pc@C=2.6pc{
 & & \Bbbk^3 \ar@/^/[d]^{B_{1,3}} & & \\
\Bbbk^3 \ar@/^/[rr]^{B_{3,3}} & & \Bbbk^6 \ar@/^/[u]^{A_{1,3}} \ar@/^/[ll]^{A_{3,3}} \ar@/^/[rr]^{A_{2,3}} & & \Bbbk^3 \ar@/^/[ll]^{B_{2,3}}
}}$$
where
\begin{gather*}
    A_{1,3} = A_{2,3} = A_{3,3} = \left(\begin{array}{@{\,}c|c@{\,}} I_3 & O_3 \end{array}\right), \\[2ex]
    B_{1,3} = \left(\begin{array}{@{\,}c@{\,}} O_3 \\ \hline I_3 \end{array}\right), \qquad
    B_{2,3} = \left(\begin{array}{@{\,}ccc@{\,}} && \\ \multicolumn{3}{c}{O_3} \\ && \\ \hline \alpha&1&0 \\ 0&\alpha&1 \\ 0&0&\alpha \end{array}\right), \qquad
    B_{3,3} = \left(\begin{array}{@{\,}ccc@{\,}} && \\ \multicolumn{3}{c}{O_3} \\ && \\ \hline 1+\alpha&1&0 \\ 0&1+\alpha&1 \\ 0&0&1+\alpha \end{array}\right).
\end{gather*}
In \cite{Name_Name_2026}, $X_n$ is defined for all $n\geq 1$, but only the first three are actually used. The three of them are indecomposable non-projective $\Lambda$-modules \cite[Lem.~3.1]{Name_Name_2026} and we point out the relations
\begin{equation} \label{AB_zero}
    A_{i,n}B_{j,n} = 0
\end{equation}
for every $i,j\in \{1,2,3\}$ and $n=1,2,3$, which do not come from the mesh relations, as the paths $b_ja_i$ with $i\neq j$ are nonzero in $\Lambda$.

The last module we need is the projective $P_0 = e_0 \Lambda$. A path basis of its vector spaces is
\begin{align*}
    (P_0)_0 &= \langle e_0, x, y, yx\rangle, & (P_0)_1 &= \langle a_1, ya_1\rangle, \\
    (P_0)_2 &= \langle a_2, ya_2\rangle, & (P_0)_3 &= \langle a_3, xa_3\rangle,
\end{align*}
where
$$ x = a_2b_2, \qquad y = a_3b_3. $$
In particular, since $x^2= 0$, $y^2= 0$, and $(x+y)^2 = (a_1b_1)^2 = 0$, we get
$$ 0 = (x+y)^2 = x^2+xy+yx+y^2 = xy+yx $$
and, therefore, the endomorphism algebra of $P_0$ is
\begin{equation} \label{end_p0}
    \End(P_0) \cong e_0 \Lambda e_0 \cong \Bbbk[x,y]/(x^2,y^2),
\end{equation}  
a commutative local algebra of dimension $4$ over $\Bbbk$.

Write $L_{i,n}$ for the lower block of $B_{i,n}$, so that $L_{1,n}=I_n$ and $L_{3,n}=I_n+L_{2,n}$, all three being invertible. The band modules admit the following period-one projective resolution.

\begin{lemma}[{\cite[Lem.~3.2]{Name_Name_2026}}] \label{period_one}
    For $n=1,2,3$ there is a short exact sequence
    \begin{equation} \label{period_one_sequence}
        0 \longrightarrow X_n \overset{i_n}{\longrightarrow} P_0^n \overset{q_n}{\longrightarrow} X_n \longrightarrow 0.
    \end{equation}
    In the bases above, write the central coordinates of $P_0^n$ as $(u_1,u_2,u_3,u_4)\in (\Bbbk^n)^4$ and its $i$-th leaf coordinates as $(v_i,w_i)\in (\Bbbk^n)^2$. Then
    \begin{align*}
        (q_n)_0(u_1,u_2,u_3,u_4) &= (u_1, L_{2,n}u_2+L_{3,n}u_3), & (q_n)_i(v_i,w_i) &= v_i, \\
        (i_n)_0(r,t) &= (0,L_{3,n}r,L_{2,n}r,t), & (i_n)_i(z) &= (0,L_{i,n}z).
    \end{align*}
\end{lemma}

Splicing \eqref{period_one_sequence} with itself, we obtain the endomorphism
\begin{equation} \label{d_n}
    d_n = i_n q_n \colon P_0^n \longrightarrow P_0^n,
\end{equation}
which satisfies $d_n^2=0$ and $\ker d_n = \operatorname{im} d_n = i_n(X_n)$. Under the identification \eqref{end_p0}, writing
$$ p = (1+\alpha)x+\alpha y, $$
we have
$$ d_1 = p, \qquad
d_2 = \begin{pmatrix} p & x+y \\ 0 & p \end{pmatrix}, \qquad
d_3 = \begin{pmatrix} p & x+y & 0 \\ 0 & p & x+y \\ 0 & 0 & p \end{pmatrix}. $$

Let $E_1^{\bullet}$ be the acyclic complex
$$ \dots \to P_0 \overset{d_1}{\longrightarrow} P_0 \overset{d_1}{\longrightarrow} P_0 \to \dots, $$
so that
$$ Z^0(E_1^{\bullet}) = \ker d_1 = i_1(X_1)\cong X_1. $$
Let now
$$ \tilde{\varepsilon}\colon E_1^{\bullet}\longrightarrow E_1^{\bullet} $$
be the morphism of complexes all of whose components are $\alpha y\in \End(P_0)$, which commutes with the differential because $\End(P_0)$ is commutative, and let
$$ \varepsilon = Z^0(\tilde{\varepsilon})\colon X_1\longrightarrow X_1 $$
be the induced endomorphism. Under the identification of $X_1$ with $\ker d_1$ given by $i_1$, it is the morphism used in \cite{Name_Name_2026},
$$ \varepsilon_0 = \begin{pmatrix} 0 & 0 \\ 1 & 0 \end{pmatrix}, \qquad \varepsilon_i = 0. $$

The pretriangulated structure $\Delta$ of \cite{Name_Name_2026} has Heller parameter
$$ \Theta = (\id+\zeta)\Theta_0\colon \id \Rightarrow \Omega^{-3}, $$
where $\zeta$ is a natural transformation of the identity functor of $\underline{\Mod}(\Lambda)$ whose component at $X_1$ is $\varepsilon$ and which vanishes at every indecomposable not isomorphic to $X_1$, such that $\id+\zeta$ is a triangulated natural automorphism \cite[Lem.~4.1]{Name_Name_2026}. Let $\theta\colon \id \Rightarrow [3]$ be the natural isomorphism on $\underline{\Ac}(\mathscr{T})$ associated to $\Delta$. We also denote by $\zeta$ the natural transformation of the identity functor of $\underline{\Ac}(\mathscr{T})$ corresponding to $\zeta$ under $Z^0$, so that $\theta = (\id+\zeta)\theta_0$ and the component of $\zeta$ at $E_1^{\bullet}$ is the class of $\tilde{\varepsilon}$.

The main theorem of \cite{Name_Name_2026} is that $\Delta$ does not satisfy Neeman's cone criterion. The proof given in that paper is computer-assisted, since it relies on exact linear algebra over $\mathbb{F}_4$: the Heller endpoint of the cone is computed at one filler and along a basis of the eight-dimensional space of fillers. We will now see that the naturality defect product of $\theta$ has a non-vanishing instance. This gives a proof that $\Delta$ does not satisfy Neeman's criterion which can be carried out entirely by hand.

Put
$$ t_2 = \begin{pmatrix} p & p+x+y \\ 0 & p \end{pmatrix}, \qquad
t_3 = \begin{pmatrix} p & p+x+y & x+y \\ 0 & p & p+x+y \\ 0 & 0 & p \end{pmatrix}, $$
and consider the $3$-periodic acyclic complexes
$$ E_2^{\bullet}\colon \; P_0^2 \overset{t_2}{\to} P_0^2 \overset{d_2}{\to} P_0^2 \overset{d_2}{\to} P_0^2, \qquad
E_3^{\bullet}\colon \; P_0^3 \overset{d_3}{\to} P_0^3 \overset{d_3}{\to} P_0^3 \overset{t_3}{\to} P_0^3 $$
where the first displayed term sits in degree $0$. These are exact triangles with respect to $\theta_0$ and $\theta$ \cite[Lem.~5.1 and Lem.~5.2]{Name_Name_2026}. This fact is good for context but irrelevant to our proof. We now consider the sequence of $3$-periodic morphisms $E_2^{\bullet} \to E_3^{\bullet} \to E_1^{\bullet}$ given by
$$\xymatrix@C=2.8pc@R=2.4pc{
P_0^2 \ar[r]^{t_2} \ar[d]_{a^0} & P_0^2 \ar[r]^{d_2} \ar[d]_{a^1} & P_0^2 \ar[r]^{d_2} \ar[d]_{a^1} & P_0^2 \ar[d]_{a^0} \\
P_0^3 \ar[r]^{d_3} \ar[d]_{\pi} & P_0^3 \ar[r]^{d_3} \ar[d]_{\pi} & P_0^3 \ar[r]^{t_3} \ar[d]_{\pi} & P_0^3 \ar[d]_{\pi} \\
P_0 \ar[r]^{d_1} & P_0 \ar[r]^{d_1} & P_0 \ar[r]^{d_1} & P_0
}$$
where
$$ a^0 = \begin{pmatrix} 1&0 \\ 0&1 \\ 0&0 \end{pmatrix}, \qquad
a^1 = \begin{pmatrix} 1&1 \\ 0&1 \\ 0&0 \end{pmatrix}, \qquad
\pi = \begin{pmatrix} 0&0&1 \end{pmatrix}. $$
Denote by $\tilde{a}=(a^0,a^1,a^1)$ and $\tilde{g}= (\pi,\pi,\pi)$ these morphisms, whose composition $\tilde{g}\tilde{a}$ is strictly zero, and by $a$ and $g$ their classes. We will prove that the naturality defect product $\langle g,a\rangle_{\theta}$ does not vanish and is therefore an explicit obstruction to (TR4).
Since $\Delta_0$ is triangulated, $\langle g,a\rangle_{\theta_0}$ vanishes by Theorem \ref{main_theorem}. Hence, by part $(5)$ of Proposition \ref{product_props}, $\langle g,a\rangle_{\theta}$ vanishes if and only if $\langle g,a\rangle_{\id+\zeta}$ does, which, by part $(2)$ of Proposition \ref{product_props} and Corollary \ref{identity_product}, is equivalent to the vanishing of $\langle g,a\rangle_{\zeta}$. It is then enough to prove that $\langle g,a\rangle_{\zeta}$ does not vanish.

\begin{remark}
    The reader may wonder why we use the projection onto the third component instead of the inclusion into the cone. The reason is simple: the complex $E_1^{\bullet}$ is much smaller and homotopy equivalent to the mapping cone. A homotopy equivalence $C^{\bullet}(\tilde{a})\to E_1^{\bullet}$ is given by the projection onto the last component. So $\tilde{g}$ is the composition of the inclusion into the cone with a homotopy equivalence and so, by part $(3)$ of Proposition \ref{product_props}, their naturality defect products vanish simultaneously. Note that the existence of such a homotopy equivalence is not surprising in view of \cite[Lem.~6.1]{Name_Name_2026}.
\end{remark}

Since $Z^0(E_2^{\bullet})\cong X_2$ and $Z^0(E_3^{\bullet})\cong X_3$, $\zeta$ vanishes at $E_2^{\bullet}$ and $E_3^{\bullet}$, so we can take zero to be a representative for $\zeta$ in those components, while in the $E_1^{\bullet}$ component we take $\tilde{\varepsilon}$. Since $\tilde{g}\tilde{a}=0$, the naturality defect product $\langle g,a\rangle_{\zeta}$ is represented by the cocycle
$$ m = w\tilde{a}\in \hom^{-1}(E_2^{\bullet},E_1^{\bullet}) $$
for some $w\in \hom^{-1}(E_3^{\bullet},E_1^{\bullet})$ such that $\partial(w) = D_{\tilde{\zeta}}(\tilde{g}) = \tilde{\varepsilon}\tilde{g}$. One can check that the $6$-periodic morphism given by
$$ w^j = \begin{pmatrix}
    0 & 1 & \gamma_j
\end{pmatrix} $$
with
$$ \gamma_0 = \alpha, \qquad \gamma_1 = 0, \qquad \gamma_2 = \alpha, \qquad \gamma_3 = 1, \qquad \gamma_4 = \alpha^2, \qquad \gamma_5 = 1 $$
works. Composition with $\tilde{a}$ shows that $m$ is given by the matrix $(0\; 1)$ in all components. 

Since $\operatorname{char}\Bbbk=2$, we have $E_1^{\bullet}[-1]=E_1^{\bullet}$, so the cocycle $m$ is a morphism of complexes $E_2^{\bullet}\to E_1^{\bullet}$ and we may write $m' = Z^0(m)\in \Hom_{\Lambda}(X_2,X_1)$. The data of a morphism $f$ in the latter space is equivalent to a matrix of the form
$$ F = \begin{pmatrix}
    0 & \lambda_1 & 0 & 0 \\
    \lambda_2 & \lambda_3 & 0 & \lambda_1 
\end{pmatrix} $$
with $\lambda_1,\lambda_2,\lambda_3\in \Bbbk$. The associated matrix is precisely the component of $f$ at vertex $0$ seen as a morphism of representations. The other components are exactly
$$ f_i = \begin{pmatrix}
    0 & \lambda_1 
    \end{pmatrix} $$
and are fully determined by $F$. In particular,
$$ \dim_{\Bbbk} \Hom_{\Lambda}(X_2,X_1) = 3. $$
Under this identification, it is straightforward to check that $m'$ corresponds to the matrix
$$ \begin{pmatrix}
0 & 1 & 0 & 0 \\
0 & 0 & 0 & 1
\end{pmatrix}. $$
We define $\varphi\colon \Hom_{\Lambda}(X_2,X_1)\to \Bbbk$ so that $F$ as above is sent to $\lambda_1$. In particular, $\varphi(m') = 1$. \emph{Claim:} $\varphi(f)$ vanishes in all of the following cases:
\begin{itemize}
    \item $f$ factors through a projective;
    \item $f \in g' \circ \Hom_{\Lambda}(X_2,X_3)$;
    \item $f \in \Hom_{\Lambda}(X_3,X_1) \circ a'$,
\end{itemize}
where $a' = Z^0(\tilde{a})$ and $g' = Z^0(\tilde{g})$. 

For the first case, we can assume $f$ factors through the injective envelope $P_0^2$ of $X_2$. The image of the inclusion into the envelope at vertex $1$ is generated by $(0,ya_1)$ and $(ya_1,0)$. Since the action of $ya_1 = a_3b_3a_1$ on $X_1$ is zero by \eqref{AB_zero}, $f$ vanishes at vertex $1$ and, therefore, $\varphi(f) = 0$.

For the second case, let $f = g' h$ for some $h\in \Hom_{\Lambda}(X_2,X_3)$. The relations $h_iA_{i,2}=A_{i,3}h_0$ and $h_0B_{i,2}=B_{i,3}h_i$ force its component at vertex $1$ to be of the form
$$ h_1 = \begin{pmatrix}
    \mu_{11} & \mu_{12} \\
    0 & \mu_{11} \\
    0 & 0
\end{pmatrix}. $$
Moreover, one sees that $g'$ is the projection onto the third component at every vertex. Hence, $g'_1 h_1 = 0$ and so $\varphi(f) = 0$.

Finally, let $f = h a'$ for some $h\in \Hom_{\Lambda}(X_3,X_1)$. The analogous relations force its component at vertex $1$ to be of the form
$$ h_1 = \begin{pmatrix}
    0 & 0 & \nu_3
\end{pmatrix}. $$
The morphism $a'$ is the inclusion of the first two components at every vertex, so $h_1 a'_1 = 0$ and $\varphi(f) = 0$.

Consequently, the class of $m'$ does not vanish in the quotient
$$ \frac{\underline{\Hom}_{\Lambda}(X_2,X_1)}{g'\circ \underline{\Hom}_{\Lambda}(X_2,X_3) + \underline{\Hom}_{\Lambda}(X_3,X_1)\circ a'}, $$
which corresponds under $Z^0$ to the quotient where $\langle g,a\rangle_{\zeta}$ lives. Since $Z^0$ is an equivalence, $\langle g,a\rangle_{\zeta}$ does not vanish either.

We have thus proved the following result.

\begin{theorem} \label{pi_d4_theorem}
    The naturality defect product $\langle g,a\rangle_{\theta}$ is a local obstruction to (TR4). Consequently, $(\proj \Lambda, \id, \Delta)$ is a pretriangulated category which is not triangulated, recovering \cite[Thm.~6.4]{Name_Name_2026}.
\end{theorem}

\begin{remark}
    All the spaces involved are small. Indeed,
    $$ \dim_{\Bbbk} \Hom_{\Lambda}(X_2,X_1) = 3, \quad \dim_{\Bbbk} \Hom_{\Lambda}(X_2,X_3) = 8, \quad \dim_{\Bbbk} \Hom_{\Lambda}(X_3,X_1) = 4, $$
    the second one being the dimension of the space of fillers in \cite{Name_Name_2026}. Moreover, the morphisms $X_2\to X_1$ factoring through a projective are exactly those with $\lambda_1=\lambda_2=0$, so that $\underline{\Hom}_{\Lambda}(X_2,X_1)$ has dimension $2$, and both summands of the indeterminacy are the line $\lambda_1=0$. The quotient above is therefore one-dimensional and the class of $m'$ is actually a generator.
\end{remark}

\begin{remark} \label{not_gemstone}
    The obstruction computed above is not gemstone-generated. Indeed, $\zeta$ is supported on $X_1$, which is not a gemstone since it is $1$-periodic by Lemma \ref{period_one}, contradicting (G1).
\end{remark}

\section{The \texorpdfstring{$n$}{n}-angulated case} \label{sec:n-angulated}

Our main result extends to the more general context of $n$-angulated categories. This is completely straightforward, so we will just state the analogues of the results of Section \ref{sec:octahedral} without proofs.

We start by recalling the definition of an $n$-angulated category following \cite{n-angulated_categories}.

\begin{definition}
    A \emph{pre-$n$-angulated category} is an additive category $\mathscr{F}$ equipped with an autoequivalence $\Sigma\colon \mathscr{F}\to \mathscr{F}$ and a class $\pentago$ of $n$-$\Sigma$-sequences, 
    $$ X^{1} \overset{\alpha^1}{\longrightarrow} X^{2} \overset{\alpha^2}{\longrightarrow} \dots \overset{\alpha^{n-1}}{\longrightarrow} X^{n} \overset{\alpha^n}{\longrightarrow} \Sigma X^1, $$
    called \emph{distinguished $n$-$\Sigma$-sequences}, such that the following axioms hold:
    \begin{itemize}
        \item[(F1)] The class $\pentago$ is closed under direct sums and direct summands. For every object $X\in \mathscr{F}$, the trivial $n$-$\Sigma$-sequence
        $$X \overset{\id}{\longrightarrow} X \to 0 \to \dots \to 0 \to \Sigma X$$
        is in $\pentago$. For every morphism $\alpha\colon X\to Y$, there is a distinguished $n$-$\Sigma$-sequence whose first morphism is $\alpha$.
        \item[(F2)] An $n$-$\Sigma$-sequence belongs to $\pentago$ if, and only if, its left rotation
        $$ X^{2} \overset{\alpha^2}{\longrightarrow} X^3 \overset{\alpha^3}{\longrightarrow} \dots \overset{\alpha^n}{\longrightarrow} \Sigma X^1 \overset{(-1)^n \Sigma \alpha^1}{\longrightarrow} \Sigma X^2 $$
        belongs to $\pentago$. 
        \item[(F3)] Every commutative square between distinguished $n$-$\Sigma$-sequences can be completed into a morphism of $n$-$\Sigma$-sequences
        $$ \xymatrix{ X^1 \ar[r]^{\alpha^1} \ar[d]_{\varphi^1} & X^2 \ar[r]^{\alpha^2} \ar[d]_{\varphi^2} & X^3 \ar[r]^{\alpha^3} \ar@{.>}[d]_{\varphi^3} & \dots \ar[r]^{\alpha^{n-1}} & X^n \ar[r]^{\alpha^n} \ar@{.>}[d]_{\varphi^n} & \Sigma X^1 \ar[d]_{\Sigma \varphi^1} \\ 
        Y^1 \ar[r]^{\beta^1} & Y^2 \ar[r]^{\beta^2} & Y^3 \ar[r]^{\beta^3} & \dots \ar[r]^{\beta^{n-1}} & Y^n \ar[r]^{\beta^n} & \Sigma Y^1 }$$
    \end{itemize}
    Furthermore, we say that a pre-$n$-angulated category is \emph{$n$-angulated} if it satisfies:
    \begin{itemize}
        \item[(F4)] For every commutative square as in (F3), the completion can be chosen so that the mapping cone sequence
        $$X^2 \oplus Y^1 \longrightarrow X^3 \oplus Y^2 \longrightarrow \dots \longrightarrow \Sigma X^1 \oplus Y^n \longrightarrow \Sigma X^2 \oplus \Sigma Y^1,$$
        whose $j$-th morphism is $\begin{psmallmatrix} -\alpha^{j+1} & 0 \\ \varphi^{j+1} & \beta^{j} \end{psmallmatrix}$ for $j<n$, the last one being $\begin{psmallmatrix} -\Sigma \alpha^{1} & 0 \\ \Sigma\varphi^{1} & \beta^{n} \end{psmallmatrix}$, is a distinguished $n$-$\Sigma$-sequence.
    \end{itemize}
\end{definition}

Notice that the axiom (F4) is the direct generalization of Neeman's cone criterion. One could also give an equivalent axiom in terms of generalized octahedra, see \cite{axioms_n-angulated}. A natural generalization of Heller's Theorem is \cite[Prop.~3.4]{n-angulated_categories}. Here we write a version of this result in terms of the category $\underline{\Ac}(\mathscr{F})$, as we did in Section \ref{sec:heller} for the triangulated case. A complex $X^{\bullet}$ over $\mathscr{F}$ is \emph{$n$-$\Sigma$-periodic} if $X^{j+n}=\Sigma X^j$ and $d^{j+n}_{X}=\Sigma d^j_{X}$ for every $j\in \mathbb{Z}$.

\begin{theorem}[{\cite[Prop.~3.4]{n-angulated_categories}}] \label{heller_n-angulated}
    Let $\mathscr{F}$ be an additive category with split idempotents admitting a pre-$n$-angulated structure and let $\Sigma\colon \mathscr{F}\to \mathscr{F}$ be an autoequivalence. Then, there is a bijection between the set of pre-$n$-angulated structures on $\mathscr{F}$ with suspension functor $\Sigma$ and the set of triangulated natural isomorphisms $(\Sigma,\id)\cong ([n],(-\id)^n)$ on $\underline{\Ac}(\mathscr{F})$. Moreover, the distinguished $n$-$\Sigma$-sequences in the pre-$n$-angulated structure associated to a natural isomorphism $\delta$ correspond to the $n$-$\Sigma$-periodic complexes $X^{\bullet}$ such that the morphism of complexes $\Sigma X^{\bullet}\to X^{\bullet}[n]$ given in degree $j$ by
    $$\tilde{\delta}_{X^{\bullet}}^j =(-1)^{n j} \id_{\Sigma X^j}$$
    is a representative for $\delta_{X^{\bullet}}$.
\end{theorem}

Here, $\delta$ being a triangulated natural isomorphism means that $\delta_{X^{\bullet}[1]}=(-1)^n\delta_{X^{\bullet}}[1]$ for every acyclic complex $X^{\bullet}$.

From now on, $\mathscr{F}$ is a pre-$n$-angulated category with split idempotents. In view of the bijection above, it has an associated natural isomorphism $\delta\colon H^*(\Sigma)\Rightarrow H^*([n])$, which is a graded natural transformation as in Section \ref{sec:octahedral}, and we can define the naturality defect product with respect to $\delta$ as in Definition \ref{delta-product}.

\begin{definition} \label{n-exact_comp}
    We say that an acyclic complex $X^{\bullet}$ is an \emph{exact $n$-gon} with respect to the natural isomorphism $\delta$ if it is $n$-$\Sigma$-periodic and the morphism of complexes $\Sigma X^{\bullet}\to X^{\bullet}[n]$ given in degree $j$ by
     $$\tilde{\delta}_{X^{\bullet}}^j =(-1)^{n j} \id_{\Sigma X^j}$$
     is a representative for $\delta_{X^{\bullet}}$.
\end{definition}

\begin{proposition} \label{n-operation_to_delta}
    Let $A^{\bullet}$ and $B^{\bullet}$ be acyclic complexes, with fixed representatives $\tilde{\delta}_{A^{\bullet}}$ and $\tilde{\delta}_{B^{\bullet}}$ of $\delta$, and take $\tilde{\delta}_{A^{\bullet}[1]}:=(-1)^n\tilde{\delta}_{A^{\bullet}}[1]$ as representative of $\delta_{A^{\bullet}[1]}$. Let $k=[\kappa]\in \underline{\Ac}(\mathscr{F})(A^{\bullet},B^{\bullet})$ and suppose that the composition
    $$A^\bullet \overset{k}{\longrightarrow} B^{\bullet} \overset{i}{\longrightarrow} C^{\bullet}(\kappa)$$
    satisfies $\langle i,k\rangle_{\delta}=0$. Then, there is some $x\in \hom^{-1}(\Sigma A^{\bullet},B^{\bullet}[n])$ satisfying $D_{\tilde{\delta}}(\kappa) = \partial(x)$ and such that the morphism of complexes $\Sigma C^{\bullet}(\kappa)\to C^{\bullet}(\kappa)[n]$ given in degree $j$ by
    \begin{equation} \label{n-forma}
        \begin{pmatrix} \tilde{\delta}_{A^{\bullet}[1]}^j & 0 \\
    x^{j+1} & \tilde{\delta}_{B^{\bullet}}^j
    \end{pmatrix}
    \end{equation}
    is a representative of $\delta_{C^{\bullet}(\kappa)}$.
\end{proposition}

\begin{proposition} \label{n-change_of_cone}
    Suppose $k\in \underline{\Ac}(\mathscr{F})(A^{\bullet},B^{\bullet})$ is a morphism between exact $n$-gons with an $n$-$\Sigma$-periodic representative $\kappa= (\kappa^{1},\dots, \kappa^{n-1}, \kappa^n) $ such that the composition
    $$A^\bullet \overset{k}{\longrightarrow} B^{\bullet} \overset{i}{\longrightarrow} C^{\bullet}(\kappa)$$
    satisfies $\langle i,k\rangle_{\delta}=0$. Then, $k$ has another $n$-$\Sigma$-periodic representative of the form $\psi = (\kappa^{1},\dots, \kappa^{n-1}, \psi^n) $ such that $C^{\bullet}(\psi)$ is an exact $n$-gon.
\end{proposition}

\begin{remark}
    Notice that, although (F4) allows us to change every component of the completion $\kappa$ other than the first two, the previous result shows that, if we already have a completion, we only need to change the last component in order to get an exact cone.
\end{remark}

\begin{theorem} \label{n-main_theorem}
    Let $\mathscr{F}$ be as above. The following are equivalent:
    \begin{enumerate}
        \item $\mathscr{F}$ is $n$-angulated.
        \item For every pair $A^{\bullet}, B^{\bullet}$ of exact $n$-gons and for every $k\in \underline{\Ac}(\mathscr{F})(A^{\bullet},B^{\bullet})$, the product $\langle i,k\rangle_{\delta}$ vanishes for some $n$-$\Sigma$-periodic representative of $k$, where $i$ denotes the class of the inclusion of $B^{\bullet}$ into the corresponding mapping cone.
        \item For every pair  $A^{\bullet}, B^{\bullet}$ of exact $n$-gons and for every $k\in \underline{\Ac}(\mathscr{F})(A^{\bullet},B^{\bullet})$, the product $\langle g,k\rangle_{\delta}$ vanishes for any morphism $g$ in $\underline{\Ac}(\mathscr{F})$ such that the naturality defect product is defined.
        \item The naturality defect product $\langle g,f\rangle_{\delta}$ vanishes for every pair of morphisms $f,g$ in $\underline{\Ac}(\mathscr{F})$ such that $gf=0$.
    \end{enumerate}
\end{theorem}

Examples of pre-$n$-angulated categories which are not $n$-angulated can be constructed using smaller examples. For instance, the first example for $n>3$ described in \cite{He_He_Zhou_2026}  was constructed from the example in \cite{Che26v1} for $n=3$ over $\proj \Pi(\mathbb{A}_5)$. This example was only constructed for $n=9$. Using our results, we show a way of constructing such examples.

Let $\delta\colon H^*(\Sigma)\Rightarrow H^*([n])$ and $\varepsilon\colon H^*(\Sigma')\Rightarrow H^*([m])$ be the natural isomorphisms associated to a pre-$n$-angulated structure and to a pre-$m$-angulated structure on $\mathscr{F}$. Then the horizontal composition $\varepsilon*\delta\colon H^*(\Sigma'\circ\Sigma)\Rightarrow H^*([n+m])$ is a triangulated natural isomorphism and so, by Theorem \ref{heller_n-angulated}, it defines a pre-$(n+m)$-angulated structure on $\mathscr{F}$ with suspension $\Sigma'\circ\Sigma$. Note that Theorem \ref{heller_n-angulated} applies even if we do not know a priori that $\mathscr{F}$ admits a pre-$(n+m)$-angulated structure: this hypothesis is only needed for $\Mod\mathscr{F}$ to be Frobenius, which follows from the existence of the pre-$n$-angulated structure; see \cite{n-angulated_categories}.

\begin{proposition} \label{new-pre-n-angulations}
    Let $\delta$ and $\varepsilon$ be the natural isomorphisms associated to an $n$-angulated structure and to a pre-$m$-angulated structure on $\mathscr{F}$ which is not $m$-angulated. Then the pre-$(n+m)$-angulated structure defined by $\varepsilon*\delta$ is not $(n+m)$-angulated.
\end{proposition}

\begin{proof}
    By Theorem \ref{n-main_theorem}, there are morphisms $f,g$ in $\underline{\Ac}(\mathscr{F})$ with $gf=0$ such that $\langle g,f\rangle_{\varepsilon}$ does not vanish. Since $[n]$ is an automorphism of $\underline{\Ac}(\mathscr{F})$, we can write $f=f'[n]$ and $g=g'[n]$ for some morphisms $f',g'$ with $g'f'=0$. Since $\mathscr{F}$ is $n$-angulated with the structure given by $\delta$, the product $\langle g',f'\rangle_{\delta}$ vanishes by Theorem \ref{n-main_theorem}. Proposition \ref{horizontal_props}, applied to $f'\colon A^{\bullet}\to B^{\bullet}$ and $g'\colon B^{\bullet}\to C^{\bullet}$, gives
    $$ \varepsilon_{C^{\bullet}[n]}\circ \Sigma'\langle g',f'\rangle_{\delta} + \langle g'[n],f'[n]\rangle_{\varepsilon}\circ \Sigma'\delta_{A^{\bullet}} =\langle g',f'\rangle_{\varepsilon*\delta} $$
    in the quotient group by the indeterminacy $M_{\varepsilon*\delta}(g',f')$, since $\delta$ and $\varepsilon$ are natural isomorphisms and $\Sigma'$ is an equivalence. The first summand is zero, but the second one is not, so $\langle g',f'\rangle_{\varepsilon*\delta}$ does not vanish. The claim now follows from Theorem \ref{n-main_theorem}.
\end{proof}

We can use Proposition \ref{new-pre-n-angulations} to produce infinitely many examples out of a single one.

\begin{corollary} \label{infinitely_many}
    Suppose that $(\mathscr{F},\Sigma)$ admits a triangulated structure and a pretriangulated structure which is not triangulated. Then, for every $r\geq 1$, the pair $(\mathscr{F},\Sigma^{r})$ admits a pre-$3r$-angulated structure which is not $3r$-angulated.
\end{corollary}

\begin{proof}
    Let $\delta$ and $\varepsilon$ be the natural isomorphisms associated to the two given structures, and define $\varepsilon_1=\varepsilon$ and $\varepsilon_{r+1}=\varepsilon_r*\delta$. By Proposition \ref{new-pre-n-angulations} and induction on $r$, each $\varepsilon_r$ defines a pre-$3r$-angulated structure with suspension $\Sigma^{r}$ which is not $3r$-angulated.
\end{proof}

The previous corollary applies to the examples in \cite{Che26v2} and the one in \cite{Name_Name_2026}, all covered in Section \ref{sec:examples_section}. As a final result, we observe that (F4) behaves well with respect to horizontal composition.

\begin{proposition} \label{horizontal_F4}
    Let $\delta\colon H^*(\Sigma)\Rightarrow H^*([n])$ and $\varepsilon\colon H^*(\Sigma')\Rightarrow H^*([m])$ be the natural isomorphisms associated to an $n$-angulated structure and to an $m$-angulated structure on $\mathscr{F}$. Then the pre-$(n+m)$-angulated structure defined by $\varepsilon*\delta$ is $(n+m)$-angulated, that is, it satisfies {\rm (F4)}.
\end{proposition}

\begin{proof}
    Let $A^{\bullet} \overset{f}{\to} B^{\bullet} \overset{g}{\to} C^{\bullet}$ be morphisms in $\underline{\Ac}(\mathscr{F})$ with $gf=0$. By Theorem \ref{n-main_theorem}, both $\langle g,f\rangle_{\delta}$ and $\langle g[n],f[n]\rangle_{\varepsilon}$ vanish, so both summands in Proposition \ref{horizontal_props} contain zero, and hence so does $\langle g,f\rangle_{\varepsilon*\delta}$. The result follows from Theorem \ref{n-main_theorem}.
\end{proof}

\newcommand{\etalchar}[1]{$^{#1}$}
\providecommand{\bysame}{\leavevmode\hbox to3em{\hrulefill}\thinspace}
\providecommand{\MR}{\relax\ifhmode\unskip\space\fi MR }
\providecommand{\MRhref}[2]{%
  \href{http://www.ams.org/mathscinet-getitem?mr=#1}{#2}
}
\providecommand{\href}[2]{#2}

\end{document}